\documentclass[12pt,a4paper]{article}
\usepackage{amsmath}
\usepackage{amssymb}
\usepackage{amsthm}
\usepackage{amsfonts}
\usepackage{latexsym}
\theoremstyle{plain}
\newtheorem{thm}{Theorem}[section]
\newtheorem{cor}{Corollary}[section]
\newtheorem{lem}{Lemma}[section]
\newtheorem{prop}{Proposition}[section]
\theoremstyle{definition}
\newtheorem{defn}{Definition}[section]
\newtheorem{exmp}{Example}[section]

\theoremstyle{remark}

\newtheorem{rem}{Remark}[section]

\title{Local scaling asymptotics\\ in phase space and time \\ in Berezin-Toeplitz quantization}
\author{Roberto Paoletti\footnote{\noindent{\bf Address:}
Dipartimento di Matematica e Applicazioni, Universit\`a degli Studi
di Milano Bicocca, Via R. Cozzi 53, 20125 Milano,
Italy; {\bf e-mail}: roberto.paoletti@unimib.it }}
\date{}

\begin{document}

\maketitle
\begin{abstract}
 This paper deals with the local semiclassical asymptotics of a quantum evolution operator
 in the Berezin-Toeplitz scheme, when both time and phase space variables are subject to appropriate
 scalings in the neighborhood of the graph of the underlying classical dynamics. Global consequences
 are then drawn regarding the scaling asymptotics of the trace of the quantum evolution as a function
 of time. 
\end{abstract}

\section{Introduction}

\subsection{Preamble}

This paper is concerned with the local asymptotics of the quantization $\Phi_\tau^\hbar$
($\tau\in \mathbb{R}$)
of a varying Hamiltonian symplectomorphism
of a symplectic manifold, $\phi_\tau:M\rightarrow M$,
in the Berezin-Toeplitz scheme \cite{ber}, \cite{bg}, \cite{schlichenmaier},
\cite{z-index}. 
Here the asymptotics are taken in the semiclassical regime $\hbar\rightarrow 0^+$,
and \lq varying\rq\, means that, rather than considering one fixed symplectomorphisms $\phi_\tau$
at time say $\tau=\tau_0\in \mathbb{R}$,
we are are concerned with the behaviour of $\Phi_\tau^\hbar$ when $\tau-\tau_0\rightarrow 0$ 
at different rates with respect to $\hbar\rightarrow 0^+$. Thus we shall look at the 
asymptotics of the distributional kernels of $\Phi_\tau^\hbar$ when both \lq time\rq \, and \lq phase space\rq\,
variables are suitably rescaled in terms of $\hbar$.

These distributional kernels are the Berezin-Toeplitz analogues of 
quantum evolution operators, and are therefore a fundamental and natural object of study.
Broadly stated, our purpose is to investigate how their local scaling 
asymptotics and geometric concentration 
relate to the underlying classical dynamics. 

This is close in spirit
to the study of near-diagonal scaling asymptotics of equivariant Bergman-Szeg\"{o} kernels
and their ensuing analogues for
 Toeplitz operators  (see \cite{bsz}, \cite{sz}, \cite{mm}), 
and their generalization to (fixed) quantized Hamiltonian symplectomorphisms
 \cite{paoletti_intjm}; however, rather than restricting our analysis to
\lq phase space\rq\, scaling asymptotics,
here we shall also consider scaling asymptotics in the \lq time\rq\, variable.
The approach we shall take is to look at the (local) asymptotics 
of $\Phi_{\tau_0+\sqrt{\hbar}\,\tau}^\hbar$ for $\tau$ suitably bounded in terms of $\hbar$.

One motivation for including \lq time variable\rq\, scaling asymptotics is 
related to trace computations. Namely,
the semiclassical asymptotics of $\mathrm{trace}(\Phi_\tau^\hbar)$ 
are controlled by the geometry of the fixed locus $M_\tau\subseteq M$ of $\phi_\tau$, 
and certain Poincar\'{e} type data along it.
Therefore, they have 
discontinuities (both in order of growth and leading order term, say) 
in the presence of jumping behavior of $M_\tau$ as a function of $\tau$. 
It is therefore of interest to examine the nature of these discontinuities in time in 
terms of the underlying symplectic dynamics. Under mild assumptions, 
they can be 
related fairly explicitly to the critical data of $f$ on $M_{\tau_0}$.

\subsection{The symplectic context and its quantization}
\label{subsctn:symplectic_context}

Let us emphasize from the outset that 
there is a broader scope for the approach in this article than the one we shall present here,
as by introducing the appropriate microlocal framework from \cite{bg} and \cite{sz} 
one might cover
compact and integral almost K\"{a}hler manifolds. 
For the sake of simplicity, we shall however 
confine ourselves to the complex projective setting.

Let us describe our geometric set-up. 
Our archetype of a \lq classical phase space\rq\, will be a connected
complex K\"{a}hler manifold $(M,J,2\omega)$, of complex dimension $d$.
So $J$ is a complex structure on $M$, and $\omega$ is a compatible symplectic structure (thus, of type $(1,1)$),
such that $g(\cdot,\cdot)=:\omega(\cdot,J\cdot)$ is a Riemannian metric, and $h=g-i\omega$ is an
Hermitian metric. 

In the K\"{a}hler setting, the symplectic and Riemannian volume forms, $\mathcal{S}_M$ and
$\mathcal{R}_M$, coincide, and we shall generally denote them by $\mu_M=:\omega^{\wedge d}/d!$.
We shall also denote by $\mathrm{dV}_M=|\mu_M|$ the Riemannian density (and the corresponding measure)
on $(M,g)$. When dealing with general symplectic submanifolds $F\subseteq M$, 
we shall need to distinguish between $\mathcal{S}_F$ and
$\mathcal{R}_F$ (\S \ref{subsctn:volume_forms}).

Any classical observable $f\in \mathcal{C}^\infty(M;\mathbb{R})$ determines
a Hamiltonian vector field $\upsilon_f$, and the corresponding 
flow of Hamiltonian symplectomorphisms 
$$\phi^M:\tau\in \mathbb{R}\mapsto (\phi^M_\tau:M\rightarrow M).$$ In most cases,
$\phi^M$ is not family of biholomorphisms of $(M,J)$; this happens if and
only if $\upsilon_f$ is the real component of a holomorphic vector field on $(M,J)$
\cite{kn}. In this case, we shall
say that $f$ is \textit{compatible} (with the K\"{a}hler structure). 

In this framework, the subclass of \textit{quantizable} phase spaces, 
according to the Berezin-Toeplitz scheme, is given by those $(M,J,\omega)$ that are actually Hodge manifolds \cite{gh}; 
we may then assume without any essential loss that $\omega$ be cohomologically integral. 
Therefore, there exists a positive holomorphic
Hermitian line bundle $(A,h)$ on $M$, such that its unique compatible connection 
$\nabla_A$ has curvature
$\Theta=-2i\,\omega$ \cite{gh}. Often $(A,h)$ is called a \textit{quantizing line bundle} of $(M,J,\omega)$
\cite{cgr}, \cite{w}, \cite{gs}, \cite{schlichenmaier}.

For $k\ge 1$, let $H^0\left(M,A^{\otimes k}\right)$ be the finite dimensional
vector space of holomorphic sections of $A^{\otimes k}$.
Then $h$ and $\mathrm{dV}_M$  endow 
$H^0\left(M,A^{\otimes k}\right)$ 
with a natural Hermitian structure $H_k$ \cite{gh}. Namely, if
$\sigma,\,\tau\in H^0\left(M,A^{\otimes k}\right)$ set
$$
H_k(\sigma,\tau)=:\int_M\,h_m\big(\sigma(m),\tau(m)\big)\,\mathrm{dV}_M(m),
$$
where we also denote by $h$ the Hermtian structure induced on $A^{\otimes k}$ by $h$.
In Berezin-Toeplitz quantization, 
the Hilbert space 
$\mathcal{H}_k=:\left(H^0\left(M,A^{\otimes k}\right),H_k\right)$ is viewed 
as a quantization of $(M,J,\omega)$ at Planck's constant $\hbar=1/k$.
Obviously, $\mathcal{H}_k$ is a closed subspace of the Hilbert space of all $L^2$-summable
sections, $L^2\left(M,A^{\otimes k}\right)$.

The quantization of  
$f\in \mathcal{C}^\infty(M;\mathbb{R})$
should be the assignment of self-adjoint operators 
$\mathcal{T}(f)_k:\mathcal{H}_k\rightarrow \mathcal{H}_k$ for every $k\ge 1$. The 
Berezin-Toeplitz approach is to set 
$\mathcal{T}(f)_k=:\mathcal{P}_k\circ \mathcal{M}_k\circ \mathcal{P}_k$, where
$\mathcal{P}_k:L^2\left(M,A^{\otimes k}\right)\rightarrow H^0\left(M,A^{\otimes k}\right)$
is the $L^2$-orthogonal projector, and $\mathcal{M}_f$ denotes multiplication by $f$.
This choice leads expected semiclassical properties and induces a 
\textit{star product} on the algebra of observables \cite{g-star}.
One thinks of $\mathcal{T}(f)=:\bigoplus_{k\ge 0}\mathcal{T}(f)_k$ as acting 
on the $L^2$ direct sum $\mathbf{L}(A)=: \bigoplus_{k\ge 0}L^2\left(M,A^{\otimes k}\right)$.

The quantization of $\phi^M_\tau:M\rightarrow M$, on the other hand, should be a
family of unitary operators $\widetilde{\Phi}_{\tau,k}:\mathcal{H}_k\rightarrow\mathcal{H}_k$. 
For compatible Hamiltonians, the choice of $\widetilde{\Phi}_{\tau,k}$
is straightforward
(for this reason, compatible classical observables on Hodge manifolds
are alternatively dubbed \textit{quantizable} \cite{cgr}). 
The general situation is more subtle, as we now recall.

The phase flow $\phi^M_\tau$ lifts in a natural manner to flows 
$\phi^{A^{\otimes k}}_\tau:A^{\otimes k}\rightarrow A^{\otimes k}$ for every $k$;
each $\phi^{A^{\otimes k}}_\tau$ is a fiberwise linear,
metric and connection preserving diffeomorphism.
These geometric data determine unitary operators 
\begin{equation}
 \label{eqn:lifted_linear_action}
\mathcal{V}_{\tau,k}:\sigma\in L^2\left(M,A^{\otimes k}\right)\,\mapsto\,
\phi^{A^{\otimes k}}_\tau\circ \sigma\circ \phi^{M}_{-\tau}\in L^2\left(M,A^{\otimes k}\right).
\end{equation}

When $f$ is compatible, both $\phi^{A^{\otimes k}}_\tau$ and $\phi^{M}_\tau$ are biholomorphisms
for every $\tau\in \mathbb{R}$ and $k\in \mathbb{Z}$; therefore $\mathcal{V}_{\tau,k}$ preserves 
$\mathcal{H}_k\subseteq L^2\left(M,A^{\otimes k}\right)$ and induces by
restriction a unitary operator $\widetilde{\Phi}_{\tau,k}:\mathcal{H}_k\rightarrow \mathcal{H}_k$.
It is convenient to regard $\widetilde{\Phi}_{\tau,k}$ as defined
on all of $L^2\left(M,A^{\otimes k}\right)$, and vanishing on the $L^2$-orthocomplement
of $\mathcal{H}_k$; with this understanding, 
$\widetilde{\Phi}_{\tau,k}=\mathcal{P}_k\circ \mathcal{V}_{\tau,k}\circ \mathcal{P}_k$.

In general, however, $\mathcal{V}_{\tau,k}(\mathcal{H}_k)\nsubseteq \mathcal{H}_k$; to obtain an
endomorphism of $\mathcal{H}_k$, we may then consider the composition 
$\mathcal{P}_k\circ \mathcal{V}_{\tau,k}\circ \mathcal{P}_k$ as before, but the latter need not
unitary. Nonetheless, by the theory of \cite{z-index},
one does obtain a 1-parameter family of unitary operators (at least for $k\gg 0$) by setting
\begin{equation}
 \label{eqn:toeplitz_fourier_compositions}
\widetilde{\Phi}_{\tau,k}=:\mathcal{T}_{\tau,k}\circ 
\mathcal{P}_k\circ \mathcal{V}_{\tau,k}\circ \mathcal{P}_k,
\end{equation}
where $\mathcal{T}_{\tau}$ is a
canonical 1-parameter family of $S^1$-invariant zeroth order \textit{Toeplitz operators}, 
and $\mathcal{T}_{\tau,k}$ is its restriction to
$\mathcal{H}_k$. 
In the linear case, this principle goes back to Daubechies \cite{d}, who introduced Toeplitz
multipliers to construct unitary projective representations of the linear symplectic group.

An example of invariant zeroth order Toeplitz operator on $X$ is 
the quantum observable $\mathcal{T}(f)=\bigoplus_{k\ge 0}\mathcal{T}(f)_k$ introduced above; 
up to a smoothing term, 
any invariant zeroth order Toeplitz operator
has the same form, with $f\in \mathcal{C}^\infty(M)$ replaced by an asymptotic
expansion $f(k)\sim f_0+k^{-1}\,f_1+\ldots$\footnote{$f$ is real valued if and only if
$\mathcal{T}(f)$ is self-adjoint.} \cite{g-star}.

Our focus here is on the local asymptotics of $\widetilde{\Phi}_{\tau_k,k}$
with $\tau_k=\tau_0+\tau/\sqrt{k}$, where $\tau$ is suitably bounded in terms of $k$.
We shall however need to adopt a slightly different view of this picture, as we next explain.

\subsection{The unit circle bundle and the Szeg\"{o} kernel}

We follow in this work the general approach to Berezin-Toeplitz quantization from
\cite{bg}, \cite{z-index}, \cite{z}, \cite{ks}, \cite{bsz}, \cite{sz}.
Thus the natural framework for the 
present analysis is the unit circle bundle
$X\subseteq A^\vee$ in the dual line bundle to $A$, 
with its structure 
$S^1$-action $r:S^1\times X\rightarrow X$ given by scalar multiplication. 
This is a principal $S^1$-bundle
over $M$, with projection $\pi:X\rightarrow M$. Then $\nabla_A$ 
corresponds to a connection 1-form $\alpha$ on $X$, and 
$d\alpha=2\,\pi^*(\omega)$. The pair $(X,\alpha)$ is a contact manifold, and
$\mu_X=:(1/2\pi)\,\alpha\wedge\pi^*\left(\mathrm{dV}_M\right)$
is a volume form on $X$, with induced measure $\mathrm{dV}_X=|\mu_X|$. 
We shall write $L^2(X)=:L^2(X,\mathrm{dV}_X)$.

There is on $X$ a natural Riemannian structure $\widetilde{g}$. This is defined by
declaring the horizontal and vertical tangent bundles of $X$,  
$\mathcal{H}(\alpha)=:\ker (\alpha)$ and $\mathcal{V}(\pi)=:\ker(d\pi)\subseteq TX$,
to be orthogonal, $\pi$ to be a Riemannian submersion, and 
the generator $\partial/\partial \theta$ of the structure $S^1$-action
to have unit norm\footnote{$\theta$ will always denote an \lq angular\rq\, coordinate which is translated 
by the $S^1$-action}. 

One advantage of dealing with $X$ is that sections of powers of $A$ naturally and
unitarily correspond to functions on $X$, of the appropriate $S^1$-equivariance.
Namely, for any $k\in \mathbb{Z}$ consider the $k$-th isotype 
$$
\mathcal{C}^\infty(X)_k=:\left\{s\in \mathcal{C}^\infty(X;\mathbb{C}):\,s\big(r_g(x)\big)=g^k\,s(x)\,\,
\forall\,x\in X,\,g\in S^1\right\},
$$
and similarly for $L^2(X)_k\subseteq L^2(X)$. Then $$\left(\mathcal{C}^\infty(X)_k,\mathrm{dV}_X\right)
\cong \left(\mathcal{C}^\infty\left(M,A^{\otimes k}\right),H_k\right)$$ as Hilbert spaces in a natural manner \cite{z}, \cite{sz}.
%%%%%%%%%%%%%%%%%%%%%%%%%%%%%%%%%%%%%%%%%%%%%%%%%%%%%%%%%%%%%%%%%%%%%%%%%%%%%%%%%

Under the latter unitary isomorphism, holomorphic sections collectively correspond
to the \textit{Hardy space} $H(X)\subseteq L^2(X)$ of $X$. Namely, 
let $H(X)_k\subseteq \mathcal{C}^\infty(X)_k$ be the subspace corresponding to $H^0\left(M,A^{\otimes k}\right)$ under the 
isomorphism $\left(\mathcal{C}^\infty(X)_k,\mathrm{dV}_X\right)
\cong \mathcal{C}^\infty\left(M,A^{\otimes k}\right)$. 
Then $H(X)=\bigoplus_{k\ge 0}H(X)_k$ (the Hilbert space direct sum).

The $L^2$-orthogonal projector 
$\Pi:L^2(X)\rightarrow H(X)$ is called the \textit{Szeg\"{o} projector} of $X$, and has
been described as a FIO in \cite{bs} (see also the discussion in \cite{sz}). 
We have $\Pi=\bigoplus_{k\ge 0}\Pi_k$, where $\Pi_k$ (the $k$-th Fourier component of $\Pi$)
is the orthogonal projector onto $H(X)_k$. Hence $\Pi_k$ is a smoothing operator, 
with Schwartz kernel\footnote{we shall systematically use the same
symbol for an operator and its distributional kernel}
\begin{equation}
 \label{eqn:equivariant_k_szego}
\Pi_k(x,y)=\sum _{j=1}^{N_k}s_{j}^{(k)}(x)\,\overline{s_{j}^{(k)}(y)}\,\,\,\,\,\,\,\,\,(x,y\in X),
\end{equation}
where $N_k=\dim H^0\left(M,A^{\otimes k}\right)$ and $\big( s_{j}^{(k)}\big)_j$ 
is an arbitrary orthonormal basis of
$H(X)_k$. 

\subsection{The contact flow}
As mentioned already, $\phi^M$ lifts to a metric and connection
preserving fiberwise linear flow
$\phi^A$, which therefore restricts to a contact flow $\phi^X$ on $X$. 
Infinitesimally, this may be described explicitly
as follows. 
Let $\upsilon_f^\sharp$ be the horizontal
lift of $\upsilon_f$ for $\alpha$; then 
\begin{equation}
\label{eqn:contact_vector_field}
\widetilde{\upsilon}_f=:\upsilon_f^\sharp-f\,\dfrac{\partial}{\partial \theta}
\end{equation}
is the contact vector field on $X$ generating $\phi^X$. Clearly
$\phi^X_\tau:X\rightarrow X$ lifts $\phi^M_\tau:M\rightarrow M$,
and pull-back
$\left(\phi^X_{-\tau}\right)^*:L^2(X)\rightarrow L^2(X)$ is a unitary operator.
By (\ref{eqn:contact_vector_field}), $\widetilde{\upsilon}_f$ depends on $f$, and not just on $\upsilon_f$.

Since $\widetilde{\upsilon}_f$ is $S^1$-invariant, 
so is $\phi^X_\tau$. Therefore, $\left(\phi^X_{-\tau}\right)^*$ preserves every $L^2(X)_k$.
The unitary operators
\begin{equation}
 \label{eqn:pull-back_circle_bundle}
V_{\tau,k}=\left.\left(\phi^X_{-\tau}\right)^*\right|_{L^2(X)_k}:L^2(X)_k\rightarrow L^2(X)_k
\end{equation}
correspond to (\ref{eqn:lifted_linear_action}) under the isomorphism
$L^2(X)_k\cong L^2\left(M,A^{\otimes k}\right)$. If $f$ is compatible, then
$V_{\tau,k}\big(H(X)_k\big)=H(X)_k$ for every $k$ and $\tau$, and the unitary operators
$\Phi_{\tau,k}:H(X)_k\rightarrow H(X)_k$ induced by restriction
correspond to $\widetilde{\Phi}_{\tau,k}:\mathcal{H}_k\rightarrow \mathcal{H}_k$.

%%%%%%%%%%%%%%%%%%%%%%%%%%%%%%%%%%%%%%%%%%%%%%%%%%%%%%%%%%%%%%%%%%%%%%%%%%%%%%%%%%%%%%%%%%%%%%%%%%%%%%%%%%%%%
\subsection{Toeplitz operators and quantum evolution}

We adopt the definition of Toeplitz operators from \cite{bg}. In view of the geometric
structure of the wave front of $\Pi$, for any
 pseudodifferential operator $Q$ on $X$ the composition $\Pi\circ Q\circ \Pi$ is well-defined.

\begin{defn}
 \label{defn:toeplitz_operator}
A \textit{Toeplitz operator} $T:\mathcal{C}^\infty(X)\rightarrow \mathcal{C}^\infty(X)$ 
of order $\ell\in \mathbb{Z}$ on $X$ 
is a composition of the form $T=:\Pi\circ Q\circ \Pi$, where $Q$ is a pseudodifferential operator
of classical type  on $X$, of order $\ell$. The \textit{principal
symbol} $\varrho_T$ of $T$ is the restriction of the principal symbol of $Q$ to the symplectic cone sprayed
by the connection form,
$$
\Sigma=:\Big\{(x,\,r\,\alpha_x)\,:\,x\in X,\,r>0\Big\}\subseteq T^*X\setminus \{0\};
$$
this is well defined, independently of the choice of $Q$ \cite{bg}.
Thus $\varrho_T(x,r\,\alpha_x)=r^{\ell}\,\varrho_T(x,\alpha_x)$. 
\end{defn}

\begin{rem}
Toeplitz operators are clearly filtered by their order, and if $\varrho_T=0$ as an operator
of order $\ell$, then $T$ is really of order $\le \ell-1$ in obvious sense. 
\end{rem}

If $\ell\le 0$, we can alternatively view $T$ as a continuous endomorphism of $H(X)$, or as a continuous 
endomorphism of
$L^2(X)$ vanishing on the $L^2$-orthocomplement $H(X)^\perp$.  
If in addition $T$ is $S^1$-invariant (which amounts to saying that $Q$ can be chosen $S^1$-invariant),
then it restricts for every $k$ to a uniformly bounded operator
\begin{equation}
 \label{eqn:toeplitz_k}
T_k=\Pi_k\circ Q\circ \Pi_k:H(X)_k\rightarrow H(X)_k.
\end{equation}
If $T$ is zeroth order and $S^1$-invariant, then $\varrho_T$ is in a natural manner a $\mathcal{C}^\infty$
function on $M$, that we shall call the \textit{reduced symbol} of $T$.

%%%%%%%%%%%%%%%%%%%%%%%%%%%%%%%%%%%%%%%%%%%%%%%%%%%%%%%%%%%%%%%%%%%%%%%%%%%%%%%%%%%%%%%%%%%%%%%%%%%%%%%%%%
The discussion in \S \ref{subsctn:symplectic_context} surrounding (\ref{eqn:toeplitz_fourier_compositions})
can then be rephrased by saying that there
exists a canonical $\mathcal{C}^\infty$ family $T_\tau$ of invariant
Toeplitz operators on $X$ such that 
\begin{equation}
 \label{eqn:adjusted_propagator_X}
\Phi_{\tau}=:T_\tau\circ \Pi\circ \left(\phi^X_{-\tau}\right)^*\circ \Pi=T_\tau\circ \left(\phi^X_{-\tau}\right)^*\circ \Pi
:H(X)\rightarrow H(X)
\end{equation}
is unitary, at least on the orthocomplement of a finite dimensional subspace \cite{z-index}.

We are led by the previous considerations to consider operators of the general form
\begin{equation}
 \label{eqn:general_composition}
U_{\tau}=:R_\tau\circ \Pi\circ \left(\phi^X_{-\tau}\right)^*\circ \Pi=R_\tau\circ \left(\phi^X_{-\tau}\right)^*\circ \Pi
:H(X)\rightarrow H(X),
\end{equation}
where
$R_\tau$ is a $\mathcal{C}^\infty$ family of $S^1$-invariant
zeroth order Topelitz operators, with reduced symbol $\varrho_\tau=:\varrho_{R_\tau}$,
and their equivariant restrictions 
\begin{equation}
 \label{eqn:general_composition_k}
U_{\tau,k}=:R_{\tau,k}\circ \left(\phi^X_{-\tau}\right)^*\circ \Pi_k:H(X)_k\rightarrow H(X)_k.
\end{equation}
These are the general models for \lq adjusted\rq\, quantum evolution operators, according to the philosophy
from \cite{z-index} (and \cite{d} in the linear case).
Adopting the notation in (\ref{eqn:equivariant_k_szego}), the Schwartz kernel 
$U_{\tau,k}\in \mathcal{C}^\infty(X\times X)$
is 
\begin{equation}
 \label{eqn:equivariant_k_U}
U_{\tau,k}(x,y)=\sum _{j=1}^{N_k}U_\tau\big(s_{j}^{(k)}\big)(x)
\cdot\overline{s_{j}^{(k)}(y)}\,\,\,\,\,\,\,\,\,(x,y\in X).
\end{equation}

Operator kernels such as (\ref{eqn:equivariant_k_U})
are thus crucial to Berezin-Toeplitz
geometric quantization; this motivates studying their near-graph local scaling asymptotics 
and their relation to the underlying symplectic dynamics \cite{paoletti_intjm}. This may be seen as a dynamical
generalization of the near diagonal scaling asymptotics of the equivariant components of the Szeg\"{o}
kernel \cite{bsz}, \cite{sz}, \cite{mm}.

Broadly speaking, the emphasis in \cite{paoletti_intjm} was on fixed time
scaling asymptotics of the form
$U_{\tau_0,k}\big(x+\upsilon_1/\sqrt{k},x_{\tau_0}+\upsilon_2/\sqrt{k}\big)$, where
$\tau_0\in \mathbb{R}$,
$x\in X$, $x_{\tau_0}=\phi^X_{-\tau_0}(x)$, $\upsilon_1\in T_xX$, $\upsilon_2\in T_{x_{\tau_0}}X$.
In the first part of this paper, we shall build on the arguments in \cite{paoletti_intjm}
to study scaling asymptotics of the form 
$U_{\tau_k,k}\big(x+\upsilon_1/\sqrt{k},x_{\tau_0}+\upsilon_2/\sqrt{k}\big)$, where
$\tau_k=\tau_0+\tau/\sqrt{k}$ and $\tau$ satisfies certain bounds; for example, $\tau$ might
be fixed and $\neq 0$. These \lq scaled time\rq\,
asymptotics are the object of Theorem \ref{thm:local_scaling_asymptotics_time}.

In the second part, we shall apply these results to the asymptotics of the traces
of $U_{\tau_k,k}$,
$\mathrm{trace}\left(U_{\tau_k,k}\right)$, under certain assumptions on the structure of the fixed
locus of $\phi^M_{\tau_0}$ and of the critical locus of $f$ 
on it (Theorem \ref{thm:global_trace_asymptotics_tau_0_morse_bott}).

 \part{Local scaling asymptotics}
 
 We first focus on the local scaling asymptotics of the \lq quantum evolution\rq \, operators
 $U_{\tau,k}$. We begin by considering general near graph-asymptotics, in the spirit of Theorem 1.4
of \cite{paoletti_intjm} but now including time rescaling. Then we specialize to the scaling asymptotics
in the neighborhood of fixed loci. Throughout our discussion, we shall adopt the short-hand
$$
m_\tau=:\phi^M_{-\tau}(m), \,\,\,\,\,\,\,x_\tau=:\phi^X_{-\tau}(x)\,\,\,\,\,\,\,\,(x\in X,\,m\in M,
\,\tau\in \mathbb{R}).
$$

 \section{Near-graph asymptotics}

\subsection{Preliminaries}

 \subsubsection{Heisenberg local coordinates}
Our scaling asymptotics 
are best expressed in a Heisenberg local coordinate system 
(HLCS) for $X$ centered at some $x\in X$, 
\begin{equation}
\gamma_x: (\theta,\mathbf{v})\in (-\pi,\pi)\times B_{2d}(\mathbf{0},\delta)\mapsto x+ (\theta,\mathbf{v})\in X;
\end{equation}
here $B_{2d}(\mathbf{0},\delta)\subseteq \mathbb{R}^{2d}$
is the open ball centered at the origin and of radius $\delta>0$. 
The additive notation is from \cite{sz},
where HLCS's were first defined and discussed.

In a HLCS the standard $S^1$-action is expressed by a translation
in the angular coordinate $\theta$, and $\gamma_x$ comes with a built-it unitary isomorphism 
$T_xX\cong \mathbb{R}\oplus T_mM\cong \mathbb{R}\oplus \mathbb{C}^{d}$, where $m=\pi(x)$; 
the first summand $\mathbb{R}\oplus \{\mathbf{0}\}$ corresponds
to the vertical tangent space, and the second, $\{0\}\oplus \mathbb{C}^{d}$, to the horizontal one.
We shall also set $x+\mathbf{v}=:x+(0,\mathbf{v})$. 

Given a system of HLC centered at $x$, if $m=\pi(x)$ then
$\mathbf{v}\mapsto m+\mathbf{v}=\pi (x+\mathbf{v})$ is a system of \textit{preferred}
local coordinates on $M$ centered at $m$ in the sense of \cite{sz}; thus the unitary structure 
of $T_mM$ corresponds to the standard one on $\mathbb{C}^d$.
With this understanding, in the notation $x+\mathbf{v}$
we may assume either $\mathbf{v}\in T_mM$ or $\mathbf{v}\in \mathbb{R}^{2d}\cong \mathbb{C}^d$. 

%%%%%%%%%%%%%%%%%%%%%%%%%%%%%%%%%%%%%%%%%%%%%%%%%%%%%%%%%%%%%%%%%%%%%%%%%%%%%%%%%%%%%%%%%%%%%

Heisenberg local coordinates make apparent the universal nature of the near diagonal scaling
asymptotics of the equivariant Szeg\"{o} kernels $\Pi_k$ \cite{sz}, as well as
of the near graph scaling asymptotics of 
$U_{\tau_0,k}$ \cite{paoletti_intjm}. The phase of the latter is determined by certain
Poincar\'{e} type data, as we now recall.

 \subsubsection{Poincar\'{e} type data}
\label{subsctn:poincare}
As in \cite{paoletti_intjm},
the Poincar\'{e} type data that controls the phase in our asymptotic expansions 
is encoded in certain quadratic forms on the tangent bundle of $M$.
Suppose $x\in X$, and set $x_{\tau_0}=:\phi^X_{-\tau_0}(x)$,
$m=:\pi(x)$, $m_{\tau_0}=:\phi^M_{-\tau_0}(m)=\pi(x_{\tau_0})$.
Let us choose HLCS's on $X$ centered at $x$ and $x_{\tau_0}$,
respectively.
This determines unitary isomorphisms $T_mM\cong \mathbb{C}^d$, 
$T_{m_{\tau_0}}M\cong \mathbb{C}^d$, under which
the differential $d_m\phi^M_{-\tau_0}:T_mM\rightarrow T_{\tau_0}M$
is represented by a $2d\times 2d$ symplectic matrix $A$. 

\begin{defn}
 \label{defn:poicare_data}
Let $A$ be a symplectic matrix with polar decomposition $A=O\,P$. 
Thus $O$ is symplectic and orthogonal (hence unitary) and $P$ is
symplectic and symmetric positive definite. Also, let 
$$
J_0=:\left(\begin{matrix}
     0&-I_d\\
I_d&0
    \end{matrix}\right)
$$
be the matrix of the standard complex structure on $\mathbb{R}^{2d}\cong \mathbb{C}^d$.
Let us define:
$$
\mathcal{Q}_A=:I+P^2,\,\,\,\,\mathcal{P}_A=:O\,\mathcal{Q}_A^{-1}\,O^t,\,\,\,\,
\mathcal{R}_A=:O\,\left(I-P^2\right)\,\mathcal{Q}_A^{-1}\,J_0\,O^t.
$$
Then $\mathcal{Q}_A$, $\mathcal{P}_A$ and $\mathcal{R}_A$ are symmetric matrices. 
Also, we shall set:
\begin{equation}
 \label{eqn:defn_nu_A}
 \nu_A=:\sqrt{\det (\mathcal{Q}_A)}.
\end{equation}

Finally, we shall define a quadratic form $\mathcal{S}_A:\mathbb{R}^{2d}\times \mathbb{R}^{2d}\rightarrow \mathbb{C}$ by setting
\begin{equation}
 \label{eqn:defn_S_A}
\mathcal{S}_A(\mathbf{u},\mathbf{w})=:-L_A(\mathbf{u},\mathbf{w})^t\,\left[\mathcal{P}_A+\frac i2\,\mathcal{R}_A\right]\,
L_A(\mathbf{u},\mathbf{w})-i\,\omega_0(A\mathbf{u},\mathbf{w}),
\end{equation}
where $L_A(\mathbf{u},\mathbf{w})=:A\,\mathbf{u}-\mathbf{w}$,
and $\omega_0(\mathbf{r},\mathbf{s})=:-\mathbf{r}^t\,J_0\,\mathbf{s}$ is the standard symplectic structure on $\mathbb{R}^{2d}$.

\end{defn}
 
These Definitions may be transferred to our geometric setting, as follows.

\begin{defn}
\label{defn:quadratic_form_along_graph}
Let us think of $(\mathbf{u},\,\mathbf{w})\in \mathbb{R}^{2d}\times \mathbb{R}^{2d}$ as representing
$(\mathtt{u},\,\mathtt{w})\in T_mM\times T_{m_{\tau_0}}M$ in the given choice of 
HLCS's. Since
$(\mathtt{u},\,\mathtt{w})\mapsto\mathcal{S}_A(\mathbf{u},\,\mathbf{w})$ is invariant under a change
of HLCS's as above, it yields a well-defined quadratic form
$$
\mathcal{S}_{\tau_0,m}:T_mM\times T_{m_{\tau_0}}M\longrightarrow \mathbb{C}.
$$
Similarly, with $\nu_A$ as in (\ref{eqn:defn_nu_A}), we shall set
\begin{equation}
\label{eqn:determinant_nu}
\nu(\tau_0,m)=:\nu_A.
\end{equation}
\end{defn}

\begin{rem}
 In the following arguments, a choice of HLCS's will be always implicit;
we shall systematically abuse notation and avoid distinguishing between
$\mathcal{S}_{\tau_0,m},\,\mathtt{u},\,\mathtt{w}...$ 
and $\mathcal{S}_{A},\,\mathbf{u},\,\mathbf{w}...$, and similarly
for $\mathfrak{Q}_{\tau_0,m}$ and $\mathfrak{Q}_{A}$ to be introduced below.
\end{rem}

\subsection{The local scaling asymptotics}

Let us consider the generalizations of Theorems 1.2 and 1.4 in \cite{paoletti_intjm} 
to rescaled times. Given $\tau_0,\,\tau\in \mathbb{R}$ we shall set 
$\tau_k=:\tau_0+\tau/\sqrt{k}$ for $k=1,2,\ldots$.
Furthermore, we shall denote by
$$
\mathrm{dist}_X:X\times X\rightarrow \mathbb{R}\,\,\,\,\,\mathrm{and}\,\,\,\,\,
\mathrm{dist}_M:M\times M\rightarrow \mathbb{R}
$$
the Riemannian distance functions on $X$ amd $M$, respectively.
Since $\pi$ is a Riemannian submersion and its fibers are the $S^1$-orbits in $X$,
we have $\mathrm{dist}_X\left(x,S^1\cdot x'\right)=\mathrm{dist}_M\big(\pi(x),\pi(x')\big)$
$\forall \,x,x'\in X$.

\subsubsection{Off-graph rapid decay}

To begin with, we have

\begin{prop}
 \label{prop:rapid_decrease}
Suppose $\tau_0\in \mathbb{R}$ and choose 
$C,\,E,\,\,\epsilon>0$.
%\delta>0$ such that
%$\frac 12<\delta<\frac{11}{18}$.
Then, uniformly for 
$$
\mathrm{dist}_X\left(y,S^1\cdot x_{\tau_0}\right)\ge C\,k^{-7/18}
\,\,\,\,\,\,\,\,\,\mathit{and}\,\,\,\,\,\,\,\,\,
|\tau|\le E\,k^{1/9-\epsilon},
%<E\,k^{\delta-\frac 12},
$$
we have
$U_{\tau_k}(x,y)=O\left(k^{-\infty}\right)$ as $k\rightarrow +\infty$.
\end{prop}

\begin{proof}
 By Theorem 1.2 of \cite{paoletti_intjm}, if $\tau'$ varies in a bounded interval
 then
 $U_{\tau'}(x,y)=O\left(k^{-\infty}\right)$ as $k\rightarrow +\infty$, 
 uniformly for
 $\mathrm{dist}_X\left(y,S^1\cdot x_{\tau'}\right)>C'\,k^{-7/18}$. 
 Thus we need only prove
 that under the hypothesis above $\mathrm{dist}_X\left(y,S^1\cdot x_{\tau_k}\right)>C'\,k^{-7/18}$,
 for all $k\gg 0$ and some constant $C'>0$ (independent of $k$).
 
 There exists $D>0$ such that $\forall\, m'\in M$ and $\tau'\in \mathbb{R}$ we have
 $$
 \mathrm{dist}_M\left(m',m'_{\tau'}\right)\le D\,|\tau'|,
 $$
 whence in the given hypothesis
\begin{equation}
 \label{eqn:bound_distance_tau_k}
 \mathrm{dist}_M\left(m'_{\tau_0},m'_{\tau_k}\right)\le D\,|\tau|/\sqrt{k}=
O\left(k^{1/9-\epsilon-1/2}\right)=o\left(k^{-7/18}\right),
\end{equation}
uniformly for $m'\in M$.

 Now if $m=\pi(x)$, $n=\pi(y)$ we have $\mathrm{dist}_X\left(y,S^1\cdot x_{\tau'}\right)=
 \mathrm{dist}_M(n,m_{\tau'})$. Therefore,
 \begin{eqnarray*}
  %\label{eqn:distance_bound_x_y}
 \mathrm{dist}_X\left(y,S^1\cdot x_{\tau_k}\right) &=&\mathrm{dist}_M\left(n,m_{\tau_k}\right) \nonumber\\
 &\ge&\mathrm{dist}_M\left(n,m_{\tau_0}\right) -\mathrm{dist}_M\left(m_{\tau_0},m_{\tau_k}\right)>
 \dfrac{C}{2}\,k^{-7/18}
 \end{eqnarray*}
if $k\gg 0$. This completes the proof.

\end{proof}

\subsubsection{The near-graph asymptotic expansion}

Now we shall revisit Theorem 1.4 of \cite{paoletti_intjm},
and show how the line of argument in its proof may adapted 
to include scaling in the time variable, and remove the 
transversality assumption on the displacement vectors. 

\begin{thm}
 \label{thm:local_scaling_asymptotics_time}
Suppose $\tau_0\in \mathbb{R}$ and $E>0$. Consider $x\in X$ 
and let $m=:\pi(x)$.
Choose HLCS's on $X$ centered at $x$ and $x_{\tau_0}$, respectively. 
Then, uniformly in 
$(\mathbf{u},\mathbf{w},\tau)\in T_mM\times T_{m_{\tau_0}}M\times \mathbb{R}$ 
with $\|\mathbf{u}\|,\,\|\mathbf{w}\|,\,|\tau|
\le E\,k^{1/9}$,
the following asymptotic
expansion holds as $k\rightarrow+\infty$:
\begin{eqnarray*}
\lefteqn{U_{\tau_0+\frac{\tau}{\sqrt{k}}}\left(x+\dfrac{\mathbf{u}}{\sqrt{k}},
x_{\tau_0}+\dfrac{\mathbf{w}}{\sqrt{k}}\right)}\\
&\sim& e^{i\tau\,\sqrt{k}\,f(m)}\,
\dfrac{2^d}{\nu(\tau_0,m)}\,\left(\frac k\pi\right)^d
\\
&&\cdot
e^{\mathcal{S}_{\tau_0,m}\big(\mathbf{u},\tau\,\upsilon_f(m)+\mathbf{w}\big)+i\,\tau\,\omega_m\big(\upsilon_f(m),\mathbf{w}\big)}
\left[\varrho_{\tau_0}(m)+\sum_{j\ge 1}k^{-j/2}\,b_j(m,\tau,\mathbf{u},\mathbf{w})\right],
\end{eqnarray*}
where each $b_j$ is $\mathcal{C}^\infty$ and 
a polynomial in $(\tau,\mathbf{u},\mathbf{w})$ of joint degree $\mathrm{deg}(b_j)\le 3j$.
\end{thm}

\begin{proof}
Before delving into the proof, it is in order to recall the general scheme of the arguments in \cite{paoletti_intjm}.

For $\tau'\in \mathbb{R}$, let us set 
$\Pi_{\tau'}=:\left(\phi^X_{-\tau'}\right)^*\circ \Pi$; in terms of Schwartz kernels, 
$$
\Pi_{\tau'}=\left(\phi^X_{-\tau'}\times \mathrm{id}_X\right)^*(\Pi).
$$ 
Then by definition
$U_{\tau'}=R_{\tau'}\circ \Pi_{\tau'}$.
Since by hypothesis these operators are $S^1$-invariant, passing to $S^1$-isotypes we have:
\begin{equation}
 \label{eqn:isotypes}
U_{\tau',k}=R_{\tau',k}\circ \Pi_{\tau',k}=R_{\tau',k}\circ \Pi_{\tau'}
=R_{\tau'}\circ \Pi_{\tau',k}.
\end{equation}
In terms of Schwartz kernels, 
\begin{equation}
 \label{eqn:isotypes_schwartz}
U_{\tau',k}(x,y)=\int_XR_{\tau'}(x,z)\,\Pi_{\tau',k}(z,y)\,\mathrm{dV}_X(z),
\end{equation}
where $\Pi_{\tau',k}(z,y)=\Pi_k(z_{\tau'},y)$.
By Theorem 1.2 of \cite{paoletti_intjm}, $U_{\tau',k}(x,y)=O\left(k^{-\infty}\right)$
as $k\rightarrow +\infty$, 
unless $y$ belongs to a small (and shrinking) $S^1$-invariant neighborhood of $x_{\tau'}$.
Also, as argued in the proof of the same Theorem, 
only a negligible contribution to the asymptotics is lost in (\ref{eqn:isotypes_schwartz})
if integration in $z$ is restricted to a small (and shrinking)  
invariant neighborhood of $x$, so that
$z_{\tau'}$ varies in a small invariant neighborhood of $x_{\tau'}$.

In these neighborhood, perhaps up to smoothing operators irrelevant to the asymptotics,
$R_{\tau'}$ and $\Pi_{\tau'}$ may be represented as FIO's. 
In fact, by the theory of  \cite{bdm_sj} (see also the discussion in \cite{sz})
in the neighborhood of $S^1\cdot x$ and $S^1\cdot x_{\tau_0}$ we can write, respectively,
\begin{equation}
 \label{eqn:Pi_as_FIO}
\Pi\left(x',x''\right)=\int_0^{+\infty}e^{iu \psi(x',x'')}\,s\left(u,x',x''\right)\,\mathrm{d}u,
\end{equation}
and
\begin{equation}
 \label{eqn:R_as_FIO}
R_{\tau'}\left(y',y''\right)=\int_0^{+\infty}e^{it \psi(y',y'')}\,
a_{\tau'}\left(t,y',y''\right)\,\mathrm{d}t,
\end{equation}
where $\psi$ is a complex phase of positive type, 
while $s$ and $a_{\tau'}$ are semiclassical symbols of degree
$d$. More precisely, there are asymptotic expansions
\begin{eqnarray} 
 \label{eqn:asymptotic_exp_phase_szego}
 s\left(u,x',x''\right)&\sim& \sum_{j\ge 0}u^{d-j}\,s_j\left(x',x''\right),\\
 \label{eqn:asymptotic_exp_phase_toeplitz}
 a_{\tau'}\left(t,y',y''\right)&\sim& \sum_{j\ge 0}t^{d-j}\,a_{\tau' j}\left(y',y''\right).
\end{eqnarray}

Suppose now that $\tau'\sim \tau_0$, so that $x_{\tau'}\sim x_{\tau_0}$.
In HLCS's centered at $x$ and $x_{\tau_0}$, respectively, we have:
\begin{equation}
 \label{eqn:leading_semiclassical_symbols}
s_0(x,x)=\dfrac{1}{\pi^d}\,\,\,\,\mathrm{and}\,\,\,\,
a_{\tau' 0}(x_{\tau_0},x_{\tau_0})=\dfrac{1}{\pi^d}\,\varrho_{\tau'}(m).
\end{equation}

Let us now apply these considerations with $\tau'=\tau_k$, where
$\tau_k=:\tau_0+\tau/\sqrt{k}$.
Insering (\ref{eqn:Pi_as_FIO}) and (\ref{eqn:R_as_FIO}) in (\ref{eqn:isotypes_schwartz}),
and following the arguments leading to 3.4 in \cite{paoletti_ijm}, we obtain
\begin{eqnarray}
 \label{eqn:analogue_of-3.4}
\lefteqn{U_{\tau_k}\left(x+\frac{\mathbf{u}}{\sqrt{k}},x_{\tau_0}+\frac{\mathbf{w}}{\sqrt{k}}\right)}\\
&\sim&\dfrac{k^2}{2\pi}\,\int_{1/E}^{E}\,\int_{1/E}^{E}\,\int_{-\epsilon}^\epsilon
\,\int_{-\epsilon}^\epsilon\,\int_{\mathbb{C}^d} \,e^{ik\Psi_1'}\,\gamma_k(z)\,
\mathcal{A}_1'\cdot\mathcal{V}(\theta,\mathbf{v})\,\mathrm{d}t\,\mathrm{d}u\,\mathrm{d}\vartheta
\,\mathrm{d}\theta\,\mathrm{d}\mathbf{v}.\nonumber
\end{eqnarray}

Let us briefly clarify (\ref{eqn:analogue_of-3.4}).
\begin{itemize}
 \item The factor $k^2$ is due to a change of integration variables 
$t\mapsto k \,t$ and $u\mapsto k\,u$.
\item $\vartheta$ is the angular coordinate on the structure group $S^1$ of $X$.
\item $z=x+(\theta,\mathbf{v})$ is a HLCS on $X$ centered at $x$.
\item $\gamma_k(z)=:\gamma_1\left(k^{7/18}\,\|\mathbf{v}\|\right)$, where
$\gamma_1\in \mathcal{C}^\infty_0\left(\mathbb{R}^{2d}\right)$ is a bump function
$\equiv 1$ in an neighborhood of the origin. Thus integration in $\mathbf{v}$
is thus over a ball of radius $O\left(k^{-7/18}\right)$.
\item $\mathcal{A}_1'=\varrho (\theta,\vartheta)\cdot \mathcal{A}_1$, 
where $\varrho$ and $\mathcal{A}_1$ are as follows.
\begin{enumerate}
 \item $\varrho(\theta,\vartheta)$ is a bump function on $\mathbb{R}^2$, supported where 
$\|(\theta,\vartheta)\|\le \epsilon$ and  $\equiv 1$ in a neighborhood of the origin.
\item With $s$ and $a_\tau$ as in (\ref{eqn:Pi_as_FIO}) and (\ref{eqn:R_as_FIO}), we have
\begin{equation}
 \label{eqn:defn_A_1}
\mathcal{A}_1=a_{\tau_k}\left(k\,t,x+\frac{\mathbf{u}}{\sqrt{k}},z\right)\,
s\left(k\,u,r_\vartheta(z_{\tau_k}),x_{\tau_0}+\frac{\mathbf{w}}{\sqrt{k}}\right).
\end{equation}
\end{enumerate}
\item The phase $\Psi_1'$ is given by
\begin{eqnarray}
 \label{eqn:phase_psi_1'}
\Psi_1'&=&t\,\psi\left(x+\dfrac{\mathbf{u}}{\sqrt{k}},x+(\theta,\mathbf{v})\right)\\
&&+u\,\psi\left(\phi_{-\tau_k}\big(x+(\theta+\vartheta,\mathbf{v})\big),
x_{\tau_0}+\dfrac{\mathbf{w}}{\sqrt{k}}\right)-\vartheta.			\nonumber																															
\end{eqnarray}

\end{itemize}

In rescaled HLC, 
$z=x+\big(\theta, \mathbf{v}/\sqrt{k}\big)$, we obtain (see
(6.2) of \cite{paoletti_ijm}):
\begin{eqnarray}
 \label{eqn:analogue_of_6.2}
\lefteqn{U_{\tau_k,k}\left(x+\dfrac{\mathbf{u}}{\sqrt{k}},x+\dfrac{\mathbf{w}}{\sqrt{k}}\right)}\\
&\sim&\dfrac{k^{2-d}}{2\pi}\,\int_{\mathbb{C}^d}
\,\left[\int_{1/E}^E\,\int_{1/E}^E\, \int_{-\epsilon}^\epsilon  \, \int_{-\epsilon}^\epsilon
\,e^{ik\,\Psi_2}\,\mathcal{A}_2\cdot \mathcal{V}\left(\theta, \frac{\mathbf{v}}{\sqrt{k}}\right)\,
\,\mathrm{d}t\,\mathrm{d}u\,\mathrm{d}\vartheta\,\mathrm{d}\theta\right]\,\mathrm{d}\mathbf{v}
.\nonumber                                
\end{eqnarray}
Here $\Psi_2$ is $\Psi_1'$ in rescaled coordinates, and
$$
\mathcal{A}_2=:\gamma_1\left(k^{-1/9}\,\|\mathbf{v}\|\right)\cdot \mathcal{A}_1'',
$$
where $\mathcal{A}_1''$ is $\mathcal{A}_1'$ in rescaled coordinates.

In view of (\ref{eqn:bound_distance_tau_k}),  the argument in the proof of Lemma 6.1 of \cite{paoletti_ijm}
still applies, and we get:

\begin{lem}
 \label{lem:analogue_of_6_1}
There exist constants $C_1,\,C_2>0$ such that the locus where $|\theta|>C_1\,k^{-7/18}$ or
$|\vartheta|>C_2\,k^{-7/18}$ contributes negligibly to the asymptotics
of (\ref{eqn:analogue_of_6.2}).
\end{lem}

Thus we can multiply the integrand 
in (\ref{eqn:analogue_of_6.2}) by
a cut-off of the form 
$\gamma_1'\left(k^{7/18}\,(\theta,\vartheta)\right)$, with $\gamma_1'\equiv 1$ near the origin, 
without changing the asymptotics.
In rescaled angular coordinates,
$(\theta,\vartheta)\mapsto \big(\theta/\sqrt{k},\vartheta/\sqrt{k}\big)$,
this becomes
$\gamma_1'\left(k^{-1/9}\,(\theta,\vartheta)\right)$, so that 
integration in $\mathrm{d}\theta\,\mathrm{d}\vartheta$ is now over a 
ball of radius $O\left(k^{1/9}\right)$ centered at the origin
in $\mathbb{R}^2$.

We then get (\textit{cfr} (6.10) of \cite{paoletti_ijm}):
\begin{eqnarray}
 \label{eqn:analogue_of_6_10}
U_{\tau_k,k}\left(x+\frac{\mathbf{u}}{\sqrt{k}},
x_{\tau_0}+\frac{\mathbf{w}}{\sqrt{k}}\right)&=&\int_{\mathbb{C}^d}\,I_k(\tau,\mathbf{u},\mathbf{w},\mathbf{v})\,\mathrm{d}\mathbf{v},
\end{eqnarray}
where
\begin{eqnarray}
\label{eqn:inner_integral}
\lefteqn{I_k(\tau,\mathbf{u},\mathbf{w},\mathbf{v})}\\
&\sim&\dfrac{k^{1-d}}{2\pi}\,
\int_{1/E}^E\,\int_{1/E}^E\,\int_{-\infty}^{+\infty}\,\int_{-\infty}^{+\infty}\,
e^{i\sqrt{k}\,\Psi_\tau\,}\,\mathcal{A}_\tau\cdot \mathcal{V}\left(\dfrac{\theta}{\sqrt{k}},
\dfrac{\mathbf{v}}{\sqrt{k}}\right)\,\mathrm{d}t\,\mathrm{d}u\,
\mathrm{d}\vartheta\,\mathrm{d}\theta.\nonumber
\end{eqnarray}
More precisely, notation in (\ref{eqn:inner_integral}) is as follows.
First, the phase is
\begin{equation}
 \label{eqn:defn_phase_Psi_tau}
\Psi_\tau(t,\theta,u,\vartheta)=:u\,\big(\tau\,f(m)+\theta+\vartheta\big)-t\,\theta-\vartheta.
\end{equation}
The first factor in the amplitude, on the other hand, is
\begin{eqnarray}
 \label{eqn:amplitude_tau}
\mathcal{A}_\tau&=:&\exp \left(i\tau\,\omega_m\big(\upsilon_f(m),A\mathbf{v}\big)-\frac t2 \,\theta^2
-\frac u2\,\big(\tau\,f(m)+\theta+\vartheta\big)^2\right)\nonumber\\
&&\cdot\exp \Big(t\,\psi_2(\mathbf{u},\mathbf{v})+u\,\psi_2\big(A\mathbf{v}-\tau\,\upsilon_f(m),\mathbf{w}\big)\Big)\cdot \mathcal{A}';
\end{eqnarray}
here $\mathcal{A}'$ is $\mathcal{A}_2$ expressed in rescaled coordinates, times a factor of the form 
$e^{ik\,R_3}$, where we have set
\begin{equation}
 \label{eqn:3rd_order_remainder}
R_3=\mathcal{R} _3^{(t,u)}\left(\frac{\mathbf{v}}{\sqrt{k}},\frac{\mathbf{u}}{\sqrt{k}},\frac{\mathbf{w}}{\sqrt{k}},
\frac{\tau}{\sqrt{k}},\frac{\vartheta}{\sqrt{k}},
\frac{\theta}{\sqrt{k}}\right),
\end{equation}
for an appropriate function $\mathcal{R}_3(\cdot,\cdot,\cdot,\cdot,\cdot)$ vanishing to third order at the origin,
and depending on $t$ and $u$.

We can view $I_k(\tau,\mathbf{u},\mathbf{w},\mathbf{v})$ as an oscillatory integral in $\sqrt{k}$ with real phase $\Psi_\tau$.
As discussed in \cite{paoletti_intjm}, $\Psi_\tau(t,\theta,u,\vartheta)$ has a unique stationary point
$P_\tau=:\big(1,0,1,-\tau\,f(m)\big)$, which is non-degenerate and satisfies
$\Psi_\tau(P_\tau)=\tau\,f(m)$. The Hessian matrix at the critical point has determinant $1$ and signature $0$.
The Hessian matrix of $\Psi_\tau$ and its inverse at the critical point are 
\begin{equation}
 \label{eqn:inverse_hessian}
H_{P_\tau}(\Psi_\tau)=\begin{pmatrix}
0&-1&0&0\\
-1&0&1&0\\
0&1&0&1\\
0&0&1&0
                      \end{pmatrix},\,\,\,\,\,\,\,\,H_{P_\tau}(\Psi_\tau)^{-1}=\begin{pmatrix}
0&-1&0&1\\
-1&0&0&0\\
0&0&0&1\\
1&0&1&0
                      \end{pmatrix}.
\end{equation}
In particular, $H_{P_\tau}(\Psi_\tau)$ has zero signature.

Using an
integration by parts in
$\mathrm{d}t\,\mathrm{d}u$ one can see that, perhaps after disregarding a negligible contribution to the asymptotics,
the integrand in (\ref{eqn:inner_integral})  may also be assumed to be compactly supported in $(\theta,\vartheta)$
near the critical point
\cite{paoletti_intjm}.

Applying Taylor expansion, and taking into account the factor
$e^{ik\,R_3}$, in view of (\ref{eqn:asymptotic_exp_phase_szego}), (\ref{eqn:asymptotic_exp_phase_toeplitz}),
and (\ref{eqn:defn_A_1}) (with $\vartheta$ rescaled to $\vartheta/\sqrt{k}$)
we have for the amplitude in (\ref{eqn:analogue_of_6_10}) an asymptotic expansion of the form
\begin{eqnarray}
 \label{eqn:rewrite_A_tau}
\lefteqn{\dfrac{k^{1-d}}{2\pi}\cdot\mathcal{A}_\tau\cdot \mathcal{V}\left(\dfrac{\theta}{\sqrt{k}},
\dfrac{\mathbf{v}}{\sqrt{k}}\right)=
e^{\Theta(\mathbf{v},\mathbf{u},\mathbf{w},\tau,t,\theta,u,\vartheta)}\cdot \mathcal{A}'
\cdot \mathcal{V}\left(\dfrac{\theta}{\sqrt{k}},
\dfrac{\mathbf{v}}{\sqrt{k}}\right)\nonumber}\\
&\sim& k^{d+1}\,e^{\Theta(\mathbf{v},\mathbf{u},\mathbf{w},\tau,t,\theta,u,\vartheta)}\cdot 
\sum_{a\ge 0}k^{-a/2}\,\mathcal{Q}_a^{(t,u)}\left(\mathbf{v},\mathbf{u},\mathbf{w},\tau,\vartheta,\theta\right),
\end{eqnarray}
where
\begin{eqnarray}
 \label{eqn:defn_Theta}
\lefteqn{\Theta(\mathbf{v},\mathbf{u},\mathbf{w},\tau,t,\theta,u,\vartheta)=:-\frac t2 \,\theta^2
-\frac u2\,\big(\tau\,f(m)+\theta+\vartheta\big)^2}\\
&&+i\tau\,\omega_m\big(\upsilon_f(m),A\mathbf{v}\big)
+t\,\psi_2(\mathbf{u},\mathbf{v})+u\,\psi_2\big(A\mathbf{v}-\tau\,\upsilon_f(m),\mathbf{w}\big),\nonumber
\end{eqnarray}
and $\mathcal{Q}_a=\mathcal{Q}_a^{(t,u)}\left(\mathbf{v},\mathbf{u},\mathbf{w},\tau,\vartheta,\theta\right)$ is a polynomial
of degree $\le 3\,a$.

Let us temporarily use the short-hand
$$
\int=:\int_{1/E}^E\,\int_{1/E}^E\,\int_{-\infty}^{+\infty}\,
\int_{-\infty}^{+\infty}.
$$
Applying the stationary phase Lemma for real quadratic phase functions \cite{duist}, we obtain from each summand in the second line
of (\ref{eqn:rewrite_A_tau})
an asymptotic expansion of the form
\begin{eqnarray}
 \label{eqn:asymptotic_expansion_each_summand}
 \lefteqn{k^{1+d-a/2}\,
\,\int
e^{i\sqrt{k}\,\Psi_\tau\,}\,e^{\Theta(\mathbf{v},\mathbf{u},\mathbf{w},\tau,t,\theta,u,\vartheta)}
\,\,\mathcal{Q}_a^{(t,u)}\left(\mathbf{v},\mathbf{u},\mathbf{w},\tau,\vartheta,\theta\right)\,\mathrm{d}t\,\mathrm{d}u\,
\mathrm{d}\vartheta\,\mathrm{d}\theta}\nonumber\\
&\sim&\dfrac{k^{d-a/2}}{\pi^{2d}}\,
e^{i\sqrt{k}\,\tau f(m)+\mathcal{B}_\tau(m,\mathbf{u},\mathbf{w},\mathbf{v})}\nonumber\\
&&\cdot\sum_{b\ge 0}\frac{1}{b!} \,
k^{-b/2}\,\mathfrak{R}^b\left(e^{\Theta}\,\mathcal{Q}_a\right)\big(\mathbf{v},\mathbf{u},\mathbf{w},\tau,1,0,1,-\tau\,f(m)\big),
\end{eqnarray}
where $\mathfrak{R}$ is the second order differential operator defined by the inverse Hessian matrix
in (\ref{eqn:inverse_hessian}):
\begin{equation}
 \label{eqn:second_order_diff_op}
 \mathfrak{R}=:\frac i2\,\begin{pmatrix}
                \partial_t&\partial_\theta&\partial_u&\partial_\vartheta
               \end{pmatrix}\,
               \begin{pmatrix}
0&-1&0&1\\
-1&0&0&0\\
0&0&0&1\\
1&0&1&0
                      \end{pmatrix}\,
                      \begin{pmatrix}
                \partial_t\\ \partial_\theta\\ \partial_u\\ \partial_\vartheta
               \end{pmatrix}.
\end{equation}
In particular, in $\mathfrak{R}$ there are no terms of the form
$\partial_t^2$ or $\partial_u^2$; it then follows from the fact that $\mathcal{Q}_a$ has degree
$\le 3a$ and the expression (\ref{eqn:defn_Theta}) for $\Theta$ that 
\begin{eqnarray}
 %\lefteqn{
 \mathfrak{R}^b\left(e^{\Theta}\,\mathcal{Q}_a\right)\big(\mathbf{v},\mathbf{u},\mathbf{w},\tau,1,0,1,-\tau\,f(m)\big)%}\\
 %&=&
 =e^{\mathcal{B}_\tau(m,\mathbf{u},\mathbf{w},\mathbf{v})}\,
 \mathcal{S}_{a,b}\big(\mathbf{v},\mathbf{u},\mathbf{w},\tau\big),
\end{eqnarray}
where
\begin{eqnarray}
 \label{eqn:B_tau_convenient}
\lefteqn{\mathcal{B}_\tau(m,\mathbf{u},\mathbf{w},\mathbf{v})=:
\Theta\big(\mathbf{v},\mathbf{u},\mathbf{w},\tau,1,0,1,-\tau\,f(m)\big)}\\
&=&
i\tau\,\omega_m\big(\upsilon_f(m),A\mathbf{v}\big)+\psi_2(\mathbf{u},\mathbf{v})+\psi_2\big(A\mathbf{v}-\tau\,\upsilon_f(m),\mathbf{w}\big)\nonumber\\
&=&i\tau\,\omega_m\big(\upsilon_f(m),\mathbf{w}\big)
+\psi_2(\mathbf{u},\mathbf{v})+\psi_2\big(A\mathbf{v},\tau\,\upsilon_f(m)+\mathbf{w}\big),
\nonumber
\end{eqnarray}
while $\mathcal{S}_{a,b}$ is a polynomial of degree $\le 3a+2b\le 3(a+b)$ (and depending smoothly
on $m$). Given this, we obtain for (\ref{eqn:inner_integral}) an asymptotic expansion of the form
\begin{eqnarray}
 \label{eqn:asymptotic_expansion_I_tau}
\lefteqn{I_k(\tau,\mathbf{u},\mathbf{w},\mathbf{v})}\\
&\sim&\dfrac{k^d}{\pi^{2d}}\,
e^{i\sqrt{k}\,\tau f(m)+\mathcal{B}_\tau(m,\mathbf{u},\mathbf{w},\mathbf{v})}
\cdot\sum_{j\ge 0}k^{-j/2}\,\mathcal{I} _j(m,\tau,\mathbf{u},\mathbf{w},\mathbf{v}),\nonumber
\end{eqnarray}
where each
$\mathcal{I}_j$ is a polynomial in $(\tau,\mathbf{u},\mathbf{w},\mathbf{v})$
of joint degree $\le 3j$ and, recalling (\ref{eqn:leading_semiclassical_symbols}),
the leading order term is $\mathcal{I}_0(m,\tau,\mathbf{u},\mathbf{w},\mathbf{v})=\varrho_{\tau_0}(m)$ .

Let us set 
$\mathbf{v}=\mathbf{v}'+\mathbf{u}$. Then
\begin{eqnarray}
 \label{eqn:traslazione}
\lefteqn{\mathcal{B}_\tau(m,\mathbf{u},\mathbf{w},\mathbf{v})=\mathcal{B}_\tau\left(m,\mathbf{u},\mathbf{w},\mathbf{v}'+\mathbf{u}\right)}\\
&=&i\,\Big[\tau\,\omega_m\big(\upsilon_f(m),A\mathbf{v}\big)-\omega_m(\mathbf{u},\mathbf{v})-
\omega_m\big(A\mathbf{v}'+A\mathbf{u},\tau\,\upsilon_f(m)+\mathbf{w}\big)\Big] \nonumber\\
&&-\frac 12\,\left\|\mathbf{v}'\right\|^2-\frac 12\,\left\|A\mathbf{v}'+A\mathbf{u}-\mathbf{w}-\tau\,\upsilon_f(m)\right\|^2.
\nonumber\end{eqnarray}
Hence, 
$\Re \big(\mathcal{B}_\tau (m,\mathbf{u},\mathbf{w},\mathbf{v})\big)\le -(1/2)\,\left\|\mathbf{v}'\right\|^2$.
Furthermore, $\mathcal{I}_j(\tau,\mathbf{u},\mathbf{w},\mathbf{v}'+\mathbf{u})$ splits as a linear combination
of terms of the form $\tau^\alpha\,\mathbf{u}^\beta\,\mathbf{w}^\gamma\,{\mathbf{v}'}^\delta$, 
where $\alpha+|\beta|+|\gamma|+\delta\le 3j$.
After integration in $\mathrm{d}\mathbf{v}'$ each such term therefore is bounded by
$C_j\,|\tau|^\alpha\,\|\mathbf{u}\|^\beta\,\|\mathbf{w}\|^\gamma\le C_j\,k^{j/3}$, where we have used that in our range
$|\tau|,\,\|\mathbf{u}\|,\,\|\mathbf{w}\|\le E\,k^{1/9}$.
Therefore,
\begin{eqnarray}
 \label{eqn:jth_term_integrated_estimate}
\lefteqn{
\left|k^{-j/2}\,\int_{\mathbb{C}^d}e^{\mathcal{B}_\tau(m,\mathbf{u},\mathbf{w},\mathbf{v})}\,
\mathcal{I}_j(\tau,m,\mathbf{u},\mathbf{w},\mathbf{v})\,\mathrm{d}\mathbf{v}\right|
}\nonumber\\
&\le&
C_j\,k^{-j/2}\,k^{j/3}=C_j\,k^{-j/6}.
\end{eqnarray}

Furthermore, in view of (\ref{eqn:amplitude_tau}), a similar bound holds for
the $J$-th step remainder.
Since integration in $\mathrm{d}\mathbf{v}$ in (\ref{eqn:analogue_of_6_10}) is actually
over a ball centered at the origin and 
of radius $O\left(k^{1/9}\right)$
in $\mathbb{C}^d$,
the expansion (\ref{eqn:asymptotic_expansion_I_tau}) may be integrated term by term.

Going back to (\ref{eqn:analogue_of_6_10}), we conclude that there is an asymptotic expansion of the form
\begin{eqnarray}
 \label{eqn:expansion_phase_int}
\lefteqn{U_{\tau_k,k}\left(x+\frac{\mathbf{u}}{\sqrt{k}},
x_{\tau_0}+\frac{\mathbf{w}}{\sqrt{k}}\right)}\\
&\sim& \dfrac{k^d}{\pi^{2d}}\,e^{i\sqrt{k}\,f(m)}
\cdot\sum_{j\ge 0}k^{-j/2}\int_{\mathbb{C}^d}\,
e^{\mathcal{B}_\tau(m,\mathbf{u},\mathbf{w},\mathbf{v})}\mathcal{I}_j(m,\tau,\mathbf{u},\mathbf{w},\mathbf{v})\,
\mathrm{d}\mathbf{v},\nonumber
\end{eqnarray}
where under the present assumptions
the $j$-th summand is uniformly bounded by $D_j\,k^{-j/2+j/3}=D_j\,k^{-(1/6)j}$.

Using (\ref{eqn:B_tau_convenient}) and 
the arguments from (3.15) to Lemma 3.3 of \cite{paoletti_intjm}, we obtain
for the leading order term:
\begin{eqnarray}
\label{eqn:analogue_of_3.15}
\lefteqn{e^{i\sqrt{k}\,f(m)}\dfrac{k^d}{\pi^{2d}}\,\varrho_{\tau_0}(m)\,\int_{\mathbb{C}^d}\,
e^{\mathcal{B}_\tau(m,\mathbf{u},\mathbf{w},\mathbf{v})}\,
\mathrm{d}\mathbf{v}}\\
&=&e^{i\sqrt{k}\,\tau f(m)+i\tau\,\omega_m\big(\upsilon_f(m),\mathbf{w}\big)}\,\left[\dfrac{k^d}{\pi^{2d}}\,\varrho_{\tau_0}(m)\,\int_{\mathbb{C}^d}\,
e^{\psi_2(\mathbf{u},\mathbf{v})+\psi_2\big(A\mathbf{v},\tau\,\upsilon_f(m)+\mathbf{w}\big)}\,
\mathrm{d}\mathbf{v}\right]\nonumber\\
&=&e^{i\sqrt{k}\,\tau f(m)+i\tau\,\omega_m\big(\upsilon_f(m),\mathbf{w}\big)}\,\left(\frac k\pi\right)^d\,\varrho_{\tau_0}(m)\,
\dfrac{2^d}{\sqrt{\det (Q_A)}}\,
e^{\mathcal{S}_A\big(\mathbf{u},\tau\,\upsilon_f(m)+\mathbf{w}\big)}.\nonumber
\end{eqnarray}

As to the lower order terms, we have for any $j\ge 1$ that the integral in (\ref{eqn:jth_term_integrated_estimate})
is a linear combination of similar integrals, with $\mathcal{I}_j$ replaced by a monomial of the form
$\tau^\alpha\,\mathbf{u}^\beta\,\mathbf{w}^\gamma\,{\mathbf{v}}^\delta$.
Furthermore, we see from (\ref{eqn:B_tau_convenient}) that
\begin{eqnarray}
 \label{eqn:B_tau_convenient-developped}
\mathcal{B}_\tau(m,\mathbf{u},\mathbf{w},\mathbf{v})&=&
i\tau\,\omega_m\big(\upsilon_f(m),\mathbf{w}\big)-i\,\omega_0\big(\mathbf{v},H\big)\nonumber\\
%&=&\psi_2(\mathbf{u},\mathbf{v})+\psi_2\big(A\mathbf{v},\tau\,\upsilon_f(m)+\mathbf{u}\big)\\
&& -\frac 12\,F -\frac 12\,\mathbf{v}^t\,Q_A\,\mathbf{v}
+\mathbf{v}^t\cdot G
,%\nonumber
\end{eqnarray}
where
\begin{eqnarray*}
F&=&F\big(\tau\,\upsilon_f(m),\mathbf{u},\mathbf{w}\big)=:
\big\|\tau\,\upsilon_f(m)+\mathbf{w}\big\|^2+\|\mathbf{u}\|^2\\
G&=&G\big(\tau\,\upsilon_f(m),\mathbf{u},\mathbf{w}\big)=:\mathbf{u}+A^t\,\big(\tau\,\upsilon_f(m)+\mathbf{w}\big)\\
H&=&H(\tau\,\upsilon_f(m),\mathbf{u},\mathbf{w})\big)=:A^{-1}\mathbf{w}-\mathbf{u}+\tau\,A^{-1}\upsilon_f(m).
\end{eqnarray*}

Let us perform the change of variables
$\mathbf{v}=\mathbf{s}+\mathbf{r}$, with $\mathbf{r}=Q_A^{-1}\,G$. 
We obtain from (\ref{eqn:B_tau_convenient-developped}) that
\begin{eqnarray}
 \label{eqn:B_tau_convenient-developped_1}
\mathcal{B}_\tau(m,\mathbf{u},\mathbf{w},\mathbf{v})&=&
\mathcal{B}_\tau(m,\mathbf{u},\mathbf{w},\mathbf{s}+\mathbf{r})\\
&=&i\tau\,\omega_m\big(\upsilon_f(m),\mathbf{w}\big) -\frac 12\,F 
-i\,\omega_0\big(\mathbf{r},H\big)-\frac 12\,\mathbf{r}^t\,Q_A\,\mathbf{r}\nonumber\\
&&-i\,\omega_0(\mathbf{s},H)-\frac 12\,\mathbf{s}^t\,Q_A\,\mathbf{s}.
\nonumber\end{eqnarray}
Hence, if $\mathcal{F}(\cdot)$ denotes the Fourier transform operator, we have
\begin{eqnarray}
\label{eqn:leading_order_term_other}
\lefteqn{
\int_{\mathbb{R}^{2d}}e^{\mathcal{B}_\tau(m,\mathbf{u},\mathbf{w},\mathbf{v})}\,\mathrm{d}\mathbf{v}
=\int_{\mathbb{R}^{2d}}e^{\mathcal{B}_\tau(m,\mathbf{u},\mathbf{w},\mathbf{s}+\mathbf{r})}\,
\mathrm{d}\mathbf{s}}\\
&=&e^{i\tau\,\omega_m\big(\upsilon_f(m),\mathbf{w}\big)-\frac 12\,F -i\,\omega_0(\mathbf{r},H)-\frac 12\,\mathbf{r}^t\,Q_A\,\mathbf{r}}\,
\int_{\mathbb{R}^{2d}}e^{-i\,\omega_0(\mathbf{s},H)-\frac 12\,\mathbf{s}^t\,Q_A\,\mathbf{s}}\,\mathrm{d}\mathbf{s}\nonumber\\
&=&e^{i\tau\,\omega_m\big(\upsilon_f(m),\mathbf{w}\big)-\frac 12\,F -i\,\omega_0(\mathbf{r},H)-\frac 12\,\mathbf{r}^t\,Q_A\,\mathbf{r}}\,
\cdot (2\pi)^d\,\left.\mathcal{F}\left(e^{-\frac 12\,\mathbf{s}^t\,Q_A\,\mathbf{s}}\right)\right|
_{\xi=-J_0H}\nonumber\\
&=& (2\pi)^d\,\det(Q_A)^{-1}\,
e^{i\tau\,\omega_m\big(\upsilon_f(m),\mathbf{w}\big)
-\frac 12\,F -i\,\omega_0(\mathbf{r},H)-\frac 12\,\mathbf{r}^t\,Q_A\,\mathbf{r}
-\frac 12\,(J_0H)^t\,Q_A^{-1}\,(J_0H)}.\nonumber
\end{eqnarray}

Now any monomial $\tau^\alpha\,\mathbf{u}^\beta\,{\mathbf{w}}^\gamma\,\mathbf{v}^\delta$ 
with $\alpha+|\beta|+|\gamma|+|\delta|\le 3j$
is a linear
combination of monomials of the form 
$\tau^{\alpha'}\,\mathbf{u}^{\beta'}\,{\mathbf{w}}^{\gamma'}\,\mathbf{s}^{\delta'}$ with
$\alpha'+|\beta'|+|\gamma'|+|\delta'|\le 3j$. On the other hand, for any such
$(\alpha',\beta',\gamma',\delta')$ we have
\begin{eqnarray}
\label{eqn:lower_order_term_other}
\lefteqn{\int_{\mathbb{R}^{2d}}\tau^{\alpha'}\,\mathbf{u}^{\beta'}\,{\mathbf{w}}^{\gamma'}\,\mathbf{s}^{\delta'}\,
e^{\mathcal{B}_\tau(m,\mathbf{u},\mathbf{w},\mathbf{s}+\mathbf{r})}\,\mathrm{d}\mathbf{s}}\nonumber\\
&=&\tau^{\alpha'}\,\mathbf{u}^{\beta'}\,{\mathbf{w}}^{\gamma'}\,
e^{i\tau\,\omega_m\big(\upsilon_f(m),\mathbf{w}\big)
-\frac 12\,F -i\,\omega_0(\mathbf{r},H)-\frac 12\,\mathbf{r}^t\,Q_A\,\mathbf{r}}\nonumber\\
&&\cdot\int_{\mathbb{R}^{2d}}\,\mathbf{s}^{\delta'}\,e^{-i\,\omega_0(\mathbf{s},H)-\frac 12\,\mathbf{s}^t\,Q_A\,\mathbf{s}}\,\mathrm{d}\mathbf{s}\nonumber\\
&=&\tau^{\alpha'}\,\mathbf{u}^{\beta'}\,{\mathbf{w}}^{\gamma'}\,
e^{i\tau\,\omega_m\big(\upsilon_f(m),\mathbf{w}\big)
-\frac 12\,F -i\,\omega_0(\mathbf{r},H)-\frac 12\,\mathbf{r}^t\,Q_A\,\mathbf{r}}\nonumber\\
&&\cdot P_{\delta'}\left.\left(e^{-\frac 12\,\xi^t\,Q_A^{-1}\,\xi}\right)\right|_{\xi=-J_0H},
\end{eqnarray}
where $P_{\delta'}$ is a suitable differential polynomial, in the collective variables $\xi$,
of degree $|\delta'|$.
Thus (\ref{eqn:lower_order_term_other}) can be decomposed as linear combination of terms of the form
\begin{eqnarray*}
\lefteqn{\tau^{\alpha'}\,\mathbf{u}^{\beta'}\,\mathbf{w}^{\gamma'}\,H^{\delta'}\,
e^{i\tau\,\omega_m\big(\upsilon_f(m),\mathbf{w}\big)-\frac 12\,F -i\,\omega_0(\mathbf{r},H)
-\frac 12\,\mathbf{r}^t\,Q_A\,\mathbf{r}
-\frac 12\,H^t\,Q_A^{-1}\,H}}\\
&\propto&\tau^\alpha\,\mathbf{u}^\beta
\,\mathbf{w}^{\gamma'}\,H^{\delta'}
\cdot e^{i\tau\,\omega_m\big(\upsilon_f(m),\mathbf{w}\big)+
\mathcal{S}_A\big(\mathbf{u},\tau\,\upsilon_f(m)+\mathbf{w}\big)}
\end{eqnarray*}
with $\alpha'+|\beta'|+|\gamma'|+|\delta'|\le 3j$. Since $H$ is linear in $(\tau,\mathbf{u},\mathbf{w})$, 
this completes the proof of Theorem \ref{thm:local_scaling_asymptotics_time}.

\end{proof}

\section{Local scaling asymptotics near fixed loci}
\label{sctn:fixed_loci_local}
We now aim to specialize the previous result to asymptotics near the fixed locus.
To this end, we first need to introduce some more terminology and auxiliary results.

\subsection{Preliminaries}

\subsubsection{Clean and very clean periods}
Let us now make our discussion more precise. To begin with, we clarify the class of periods to
which our results will apply. 

\begin{defn}
 \label{defn:very_clean_periods}
 Let $\tau_0\in \mathbb{R}$. 
 \begin{itemize}
  \item We shall let $M_{\tau_0}=:\mathrm{Fix}\left(\phi^M_{\tau_0}\right)\subseteq M$ be the fixed locus of
  $\phi^M_{\tau_0}:M\rightarrow M$.
  \item We shall say that $\tau_0$ is a \textit{clean period} of the phase flow $\phi^M$ if
  \begin{enumerate}
   \item $M_{\tau_0}$ is a submanifold of $M$;
   \item at each $m\in M_{\tau_0}$, we have 
   $T_mM_{\tau_0}=\ker \left(d_m\phi^M_{\tau_0}-\mathrm{id}_{T_mM_{\tau_0}}\right)$.
  \end{enumerate}
  \item We shall say that $\tau_0$ is a \textit{very clean period}, or a \textit{symplectically clean period}, 
  of $\phi^M$ if
  \begin{enumerate}
   \item $\tau_0$ is a clean period of $\phi^M$;
   \item $M_{\tau_0}\subseteq M$ is a symplectic submanifold.
  \end{enumerate}
   \end{itemize}

\end{defn}

\begin{rem}
\label{rem:tangent_normal_spaces}
 If $\tau_0\in \mathbb{R}$ is a clean period of $\phi^M$, then it is very clean if and only
if, in addition, 
\begin{equation}
 \label{eqn:very_clean_periods}
\ker \left(\mathrm{id}_{T_mM}-d_m\phi^{M}_{\tau_0}\right)\cap 
\mathrm{im} \left(\mathrm{id}_{T_mM}-d_m\phi^{M}_{\tau_0}\right)
=\{\mathbf{0}\}
\end{equation}
for every $m\in M_{\tau_0}$. In fact, the latter space on the left hand side of (\ref{eqn:very_clean_periods})
is the symplectic normal space of $T_mM_{\tau_0}$ (\S 4 of \cite{dg}).
\end{rem}

Our local scaling asymptotics will apply to all
very clean periods. 

\begin{defn}
 \label{defn:tangent_normal_spaces}
Suppose that $P\subseteq M$ is a submanifold, and $T_pP\subseteq T_pM$ is its tangent subspace
at a point $p\in P$. We shall denote by $N_p^gP=\left(T_pP\right)^{\perp_g}\subseteq T_pM$ the Riemannian orthocomplement
of $T_pP$ with respect to the Riemannian metric $g$, and by $N_p^\omega P=\left(T_pP\right)^{\perp_\omega}\subseteq T_pM$ the symplectic
normal space
with respect to the symplectic structure $\omega$. 
\end{defn}

Suppose $\tau_0$ is a very clean period of $\phi^M$ and $F\subseteq M_{\tau_0}$ is a connected component
of the fixed locus. If $m\in F$, then by the above
\begin{equation}
 \label{eqn:normal_space_image}
N_m^\omega F=\mathrm{im} \left(\mathrm{id}_{T_mM}-d_m\phi^{M}_{\tau_0}\right).
\end{equation}

\begin{rem}
\label{rem:riemannian_vs_symplectic_orthocomplement}
Unless $P$ is a complex
submanifold, $N_p^gP\neq N_p^\omega P$ in general. Except when $f$ is compatible, $M_{\tau_0}$
needn't be a complex submanifold.
\end{rem}
%%%%%%%%%%%%%%%%%%%%%%%%%%%%%%%%%%%%%%%%%%%%%%%%%%%%%%%%%%%%%%%
\subsubsection{Poincar\'{e} data along clean fixed loci}

Let us specialize the considerations of \ref{subsctn:poincare}  
to the case of a fixed point of $m\in M_{\tau_0}$.
If $m=\pi(x)$ and we fix a HLCS centered at $x$, $d_m\phi^{-\tau_0}:T_mM\rightarrow T_mM$
corresponds to a symplectic matrix $A$.

\begin{defn}
\label{defn:fixed_point_case}
Restricting the quadratic form (\ref{eqn:defn_S_A})
to the diagonal subspace in 
$\mathbb{R}^{2d}\times \mathbb{R}^{2d}$, we obtain a quadratic form
$\mathfrak{Q}_A:\mathbb{R}^{2d}\rightarrow \mathbb{C}$ given by
$\mathfrak{Q}_A(\mathbf{u})=-\mathcal{S}_A(\mathbf{u},\mathbf{u})$.
If $m=m_{\tau_0}$, with the previous choices $\mathfrak{Q}_A$
clearly corresponds to the well-defined 
quadratic form on $T_mM$
\begin{equation}
 \label{eqn:quadratic_form}
\mathfrak{Q}_{\tau_0,m}(\mathbf{v})=:-\mathcal{S}_{\tau_0,m}(\mathbf{v},\mathbf{v})
\end{equation}
(Definition \ref{defn:quadratic_form_along_graph}).

\end{defn}

 \begin{rem}
 \label{rem:positive_definite}
Since $\mathcal{P}_A$ is clearly positive definite, 
$\Re \big(\mathfrak{Q}_{\tau_0,m}(\mathbf{v})\big)\ge 0$ for every $\mathbf{v}\in T_mM$. In addition, if
$\tau_0$ is very clean then by Remark \ref{rem:tangent_normal_spaces}, Definition \ref{defn:tangent_normal_spaces},
(\ref{eqn:defn_S_A}) and (\ref{eqn:normal_space_image})
there exists $C=C_{\tau_0}>0$ such that for any connected component $F$ of $M_{\tau_0}$ we have
\begin{equation}
 \label{eqn:normal_restriction_pos_def}
\Re \big(\mathfrak{Q}_{\tau_0,m}(\mathbf{v})\big)\ge C\,\|\mathbf{v}\|^2\,\,\,\mathrm{if}\,\,\,m\in F
\,\,\,\mathrm{and}\,\,\,\mathbf{v}\in N_m^\omega F;
\end{equation}
here $\|\cdot\|=\|\cdot\|_m$ is the pointwise norm for $g$ at $m\in M$. 
\end{rem} 

\begin{defn}
 \label{defn:normal_quadratic_form}
If $F$ is a connected component of $M_{\tau_0}$, and
$m\in F$, then by Remark \ref{rem:positive_definite} 
the restriction $\mathfrak{Q}_{\tau_0,m}^{\mathrm{nor}}$
of $\mathfrak{Q}_{\tau_0,m}$
to $N^\omega _mF$ has positive definite real part. We shall denote by
$\det\left(\mathfrak{Q}_{\tau_0,m}^{\mathrm{nor}}\right)$ its determinant with respect to any orthonormal
basis of $N^s_mF$ for the restricted Euclidean structure of $T_mM$. 
The square root $\det\left(\mathfrak{Q}_{\tau_0,m}^{\mathrm{nor}}\right)^{1/2}$ is 
defined in a standard manner \cite{hormander_I}.
\end{defn}

\begin{rem}
\label{rem:positive_definite_hamiltonian}
 More generally, in the same situation we can consider the quadratic expression 
$\mathcal{S}_{\tau_0,m}\big(\mathbf{u},\tau\,\upsilon_f(m)+\mathbf{u}\big)$
for $(\mathbf{u},\tau)\in N^\omega_{m}F\times \mathbb{R}$. Since $A-I_{2d}$ restricts to an isomorphism of
$N^\omega_mF$ and $\upsilon_f(m)\in T_mF$, in this case in place of 
(\ref{eqn:normal_restriction_pos_def})
we have 
\begin{eqnarray}
\label{eqn:normal_restriction_pos_def-hamiltonian}
 \lefteqn{\Re\Big(\mathcal{S}_{\tau_0,m}\big(\mathbf{u},\tau\,\upsilon_f(m)+\mathbf{u}\big)\Big)}\nonumber\\
&=&-L_A\big(\mathbf{u},\tau\,\upsilon_f(m)+\mathbf{u}\big)^t\,\mathcal{P}_A\,L_A\big(\mathbf{u},\tau\,\upsilon_f(m)+\mathbf{u}\big)\nonumber\\
&=&-\big[(A-I_{2d})\mathbf{u}-\tau\,\upsilon_f(m)\big]^t\,\mathcal{P}_A\,\big [(A-I_{2d})\mathbf{u}-\tau\,\upsilon_f(m)\big ]\nonumber\\
&\le &- C\,\left(\|\mathbf{u}\|^2+\tau^2\,\|\upsilon_f(m)\|^2\right)
\end{eqnarray}
for some $C>0$ depending only on $\tau_0$. 
\end{rem}

If $m\in M_{\tau_0}$, then $m=m_{\tau_0}$; therefore, if $\pi(x)=m$ then
$\pi(x_{\tau_0})=m_{\tau_0}=\pi(x)$, whence
$x_{\tau_0}\in S^1\cdot x$ (the $S^1$-orbit through $x$).
Thus there exists a unique $h(x)\in S^1$ such that $x_{\tau_0}=r_{h(x)}(x)$. 
If $h(x)=e^{i\vartheta_0}$, we have
$x_{\tau_0}=x+(\vartheta_0,\mathbf{0})$. Thus a HLCS centered at $x$ obviously induces a HLCS
centered at $x_{\tau_0}$, given by
$x_{\tau_0}+(\theta,\mathbf{v})=:x+(\vartheta_0+\theta,\mathbf{v})$.

If $F_1,\ldots,F_{\ell_{\tau_0}}$ are the connected components
of $M_{\tau_0}$, then $h(x)$ is constant over each $F_a$.
Let $h_a\in S^1$ be the constant value of $h(x)$ for $\pi(x)\in F_a$.
Then for any $x\in F_a$, $k\in \mathbb{N}$ and $\tau\in \mathbb{R}$ we have
\begin{eqnarray}
\label{eqn:equivariance_property_tau_a}
U_{\tau,k}\big(x+\mathbf{v},x+\mathbf{v}\big)&=&
U_{\tau,k}\left(x+\mathbf{v},h_a^{-1}\cdot (x_{\tau_0}+\mathbf{v})\right)\nonumber\\
&=&
h_a^k\cdot U_{\tau,k}\left(x+\mathbf{v},x_{\tau_0}+\mathbf{v}\right).
\end{eqnarray}

\subsection{Diagonal rapid decrease away from clean fixed loci}

For $\tau\in \mathbb{R}$, let us set
$X_\tau=:\pi^{-1}(M_\tau)$ (in general, $X_\tau$ \textit{is not} the fixed locus of
$\phi^X_\tau$, but contains it). 
Since $\pi:X\rightarrow M$ is a Riemannian submersion, we have
\begin{equation}
\label{eqn:riem_submersion}
 \mathrm{dist}_X(y,X_{\tau})=\mathrm{dist}_M\big (\pi(y),M_{\tau}\big)
\end{equation}
for any $y\in X$.

\begin{prop}
 \label{prop:rapid_decrease_off_fixed_locus}
Suppose that $\tau_0$ is a clean period of $\phi^M$, and choose 
$C,\,E,\,\,\epsilon>0$.
Uniformly in $y\in X$ and $\tau\in \mathbb{R}$ satisfying
\begin{equation}
 \label{eqn:hypothesis_prop_bound_fixed_locus}
\mathrm{dist}_X\big(y,X_{\tau_0}\big)>C\,k^{-7/18}
\,\,\,\,\,\,\,\,\,\mathit{and}\,\,\,\,\,\,\,\,\,
|\tau|<E\,k^{1/9-\epsilon},
\end{equation}
%|\tau|<E\,k^{\delta-\frac 12},
we have
$$
U_{\tau_0+\tau/\sqrt{k}}(y,y)=O\left(k^{-\infty}\right).
$$
\end{prop}

\begin{proof}
Let us set as before $\tau_k=\tau_0+\tau/\sqrt{k}$.
Adapting the argument of Proposition \ref{prop:rapid_decrease},
we are now reduced to proving that, under the previous assuptions, 
for some $D>0$ we have
\begin{equation}
 \label{eqn:lower_bound_distance_between_orbits}
\mathrm{dist}_X\left(y,S^1\cdot y_{\tau_k}\right)\ge D\,k^{-7/18}.
\end{equation}

To this end, write $n=\pi(y)$ and remark that 
$\forall \,\tau'\in \mathbb{R}$
\begin{equation}
\label{eqn:riem_submersion}
 \mathrm{dist}_X(y,X_{\tau'})=\mathrm{dist}_M(n,M_{\tau'}),\,\,\,\,\,
\mathrm{dist}_X\left(y,S^1\cdot y_{\tau'}\right)=
\mathrm{dist}_M\left(n,n_{\tau'}\right);
\end{equation}
on the other hand,
using first that $\tau_0$ is a clean period and then 
(\ref{eqn:hypothesis_prop_bound_fixed_locus}),  
we have
\begin{equation}
 \label{eqn:normal_bound_distance}
\mathrm{dist}_M(n,n_{\tau_0})\ge C_1\,\mathrm{dist}_M(n,M_{\tau_0})\ge 
(C\,C_1)\,k^{-7/18}.
\end{equation}

Putting together (\ref{eqn:riem_submersion}), (\ref{eqn:normal_bound_distance})
and (\ref{eqn:bound_distance_tau_k}),
we get if
$\mathrm{dist}_X\left(y,X_{\tau_0}\right)\ge C\,k^{-7/18}$:
\begin{eqnarray*}
 %\label{eqn:distance_tau_k}
\mathrm{dist}_X\left(y,S^1\cdot y_{\tau_k}\right)&=&\mathrm{dist}_M\big(n,n_{\tau_k})\nonumber\\
&\ge&\mathrm{dist}_M\big(n,n_{\tau_0}\big)-\mathrm{dist}_M\big(n_{\tau_k},n_{\tau_0}\big)
%\nonumber\\
%&\ge&
\ge\frac 12\,CC_1\,k^{-7/18}
\end{eqnarray*}
if $k\gg 0$.
This establishes (\ref{eqn:lower_bound_distance_between_orbits}), hence proves 
the Proposition.
\end{proof}

\subsection{The asymptotic expansion near a fixed locus}

If we specialize Theorem \ref{thm:local_scaling_asymptotics_time} to
the asymptotics near a very clean fixed locus, in view of 
(\ref{eqn:equivariance_property_tau_a}) we immediately obtain:

\begin{cor}
 \label{cor:local_scaling_asymptotics_time_fixed_point_case}
Suppose that $\tau_0$ is a very clean period of $\phi^M$
and $E>0$. Let $F\subseteq M_{\tau_0}$ be a connected
component of the fixed locus.
Then, uniformly in $x\in \pi^{-1}(F)\subseteq X_{\tau_0}$,
in the choice of a HLCS centered at $x$,
and in
$(\mathbf{u},\tau)\in N^\omega_{\pi(x)}F\times \mathbb{R}$ with
$\|\mathbf{u}\|, \,|\tau|
\le E\,k^{1/9}$, 
the following asymptotic
expansion holds as $k\rightarrow+\infty$:
\begin{eqnarray*}
%\lefteqn{
\lefteqn{U_{\tau_0+\frac{\tau}{\sqrt{k}}}\left(x+\dfrac{\mathbf{u}}{\sqrt{k}},
x+\dfrac{\mathbf{u}}{\sqrt{k}}\right)
%}\\
\sim \left(\frac k\pi\right)^d\,\varrho_{\tau_0}(m)\dfrac{2^d}{\nu(\tau_0,m)}\, e^{i\tau\,\sqrt{k}\,f(m)}}
\\
&&\,\,\,\,\,\,\,\,\,\,\,\,\,\,																															
\cdot h_a^k\,e^{i\tau\,\omega_m\big(\upsilon_f(m),\mathbf{u}\big)+
\mathcal{S}_{\tau_0,m}\big(\mathbf{u},\tau\,\upsilon_f(m)+\mathbf{u}\big)}
\,\left[1+\sum_{j\ge 1}k^{-j/2}\,b_j(m,\tau,\mathbf{u})\right],
\end{eqnarray*}
where $m=\pi(x)$, and $b_j$ is $\mathcal{C}^\infty$ and 
a polynomial in $(\tau,\mathbf{u})$ of joint degree $\mathrm{deg}(b_j)\le 3j$.
\end{cor}

We can further specialize this at a fixed time $\tau_0$, that is with $\tau=0$,
and obtain
the following, which is also a consequence of
Theorem 1.4 of \cite{paoletti_intjm}:

\begin{cor}
 \label{cor:local_fixed_time_expansion}
Under the assumptions of Corollary \ref{cor:local_scaling_asymptotics_time_fixed_point_case},   
as $k\rightarrow +\infty$ the following asymptotic expansion holds
 uniformly for $m\in F_a\subseteq M_{\tau_0}$, $x\in \pi^{-1}(x)$ and
$\|\mathbf{v}\|\le C\,k^{1/9}$ with $\mathbf{v}\in N^\omega_{m}F$:
\begin{eqnarray*}
 %\label{eqn:expansion_pao_ijm}
U_{\tau_0,k}\left(x+\dfrac{\mathbf{v}}{\sqrt{k}},x+\dfrac{\mathbf{v}}{\sqrt{k}}\right)
&\sim& \left(\frac k\pi\right)^d
\,\varrho_{\tau_0}(m)\dfrac{2^d}{\nu(\tau_0,m)}\\
&&\cdot h_a^k\, e^{-\mathfrak{Q}_{A}(\mathbf{v})}\,\left[1+\sum_{j\ge 1}k^{-j/2}\,a_j(m,
\tau_0,\mathbf{v})\right],
\nonumber
\end{eqnarray*}
where $a_j$ is a polynomial in $\mathbf{v}$ of parity $j$, and degree $\le 3j$.
\end{cor}

The statement about the parity is part of Theorem 1.4 of \cite{paoletti_intjm}.

\begin{rem}
 \label{rem_:exponential_decay}
In view of Remarks \ref{rem:positive_definite_hamiltonian} 
and \ref{rem:positive_definite},
the previous Corollaries describe an exponential decay along transverse directions to the fixed
locus.
\end{rem}

\part{Trace scaling asymptotics}

We shall now apply the local diagonal scaling asymptotics near fixed loci from
\S \ref{sctn:fixed_loci_local} to study global scaling asymptotics
for the trace of $U_{\tau,k}$ near certain well-behaved periods; now the scaling will
be solely with respect to the time variable, as it approaches a
fixed period $\tau_0$
at a controlled pace.
We shall first dwell however on the asymptotics as $k\rightarrow +\infty$ of the trace
of $U_{\tau_0,k}$.
Prior to that, we need to introduce some preliminaries on the natural volume forms and measures
on certain submanifolds of $M$.

\section{Trace asymptotics at a fixed time}

\subsection{Symplectic vs Riemannian volume forms}
\label{subsctn:volume_forms}

\subsubsection{Vector spaces}
Let $(V,J_V,\omega_V)$ be an Hermitian vector space of finite
complex dimension $d\ge 1$ (thus $J_V$ is a linear complex structure on
$V$, and there is a positive definite Hermitian product $h_V:V\times V\rightarrow \mathbb{C}$, 
for which $\omega_V=-\Im (h_V)$). Then $g_V=:\Re (h_V):V\times V\rightarrow \mathbb{R}$ is an Euclidean
product on $V$, viewed as a real vector space. 

Let $W\subseteq V$ be any oriented $r$-dimensional real vector
subspace and let
$g_W:W\times W\rightarrow \mathbb{R}$ be the Euclidean product
on $W$ induced by restriction of $g$; 
then there is an induced \textit{Euclidean volume form} 
$E_W:\bigwedge^rW\rightarrow \mathbb{R}$,
characterized by $E_W\big(w_1,\ldots,w_r)=1$ 
on any oriented orthonormal basis of $(W,g_W)$. If $(w_j)$ is any
oriented basis of $W$, then 
$E_W\big(w_1,\ldots,w_r)=\sqrt{\det \big(h(w_l,w_k)\big)}$.

Suppose that $r=2\,\ell$, and $W$ is a \textit{symplectic} vector subspace
of $(V,\omega_V)$, with restricted symplectic
form $\omega_W:W\times W\rightarrow \mathbb{R}$. Then there is
also a \textit{symplectic volume form} 
$S_W=(1/\ell!)\,\omega_W^{\wedge \ell}:
\bigwedge^rW\rightarrow \mathbb{R}$, which is characterized by
$S_W(w_1,\ldots,w_r)=1$ if $(w_j)$ is any \textit{Darboux} basis of $(W,\omega_W)$.
We shall implicitly consider any symplectic vector
space as oriented by its symplectic volume form (equivalently, by the choice of the orientation
class of any Darboux basis).

Thus on a symplectic vector subspace $W$ of $(V,J_V,\omega_V)$ we have two
naturally induced volume forms $E_W$ and $S_W$.

In particular, any \textit{complex} $\ell$-dimensional
vector subspace $W\subseteq V$ 
is also Hermitian for the metric $h_W=g_W-i\,\omega_W$ given by restriction
of $h$.
It is therefore also a symplectic vector subspace; if $(e_1,\ldots,e_l)$ is an orthonormal basis of
$(W,h_W)$ over $\mathbb{C}$, then $\big(e_1,J(e_1),\ldots,e_l,J(e_l)\big)$ is a Darboux basis
of $(W,\omega_W)$ over $\mathbb{R}$, and an oriented orthonormal basis of $(W,g_W)$.
Thus $E_W=S_W$, since they both equal one on the latter basis.

For a general symplectic vector subspace $W\subseteq V$, 
\begin{equation}
\label{eqn:comparison_volume_form_vector_space}
 S_W=\zeta(W)\,E_W
\end{equation}
for a unique $\zeta(W)>0$. Clearly, $\zeta(W)=\det \big(h(w_l,w_k)\big)^{-1/2}$ if
$(w_j)$ is an arbitrary Darboux basis of $W$.

Similarly, if  $W^s\subseteq V$ is the symplectic orthocomplement of $W$,
then $S_{W^s}=\zeta(W^s)\,E_{W^s}$.

\subsubsection{Manifolds}
Globally, on a complex $d$-dimensional K\"{a}hler manifold $(R,J_R,\omega_R)$
the symplectic volume form $\mathcal{S}_R=\omega_R^{\wedge d}/d!$ equals the
Riemannian volume form $\mathcal{E}_R$
of the oriented Riemannian manifold $(R,g_R)$, where $g_R(\cdot,\cdot)=\omega_R\big(\cdot,J_R(\cdot)\big)$; 
in this case, we shall write $\mu_R=\mathcal{S}_R=\mathcal{E}_R$, and denote the corresponding
measure by $\mathrm{dV}_R$. Obviously, 
the same applies to any complex submanifold of $R$.
  
Consider a general symplectic submanifold $\iota: Q\hookrightarrow R$,
of (real) dimension $2\,\ell$
and codimension $2\,c=2\,(d-\ell)$.
For $q\in Q$, let $T_qQ\subseteq T_qR$ be the tangent space of $Q$ at $q$, 
and $N_q^\omega Q\subseteq T_qR$
the symplectic orthocomplement of $T_qQ$.
Thus we have a symplectic direct sum $T_qR=T_qQ\oplus N_q^\omega Q$. 

Let ${\omega_Q}_q$ and
${\eta_Q}_q$ be the induced symplectic structures on $T_qQ$ and
$N_q^\omega Q$, respectively. 
This defines symplectic structures $\omega_Q$ and
$\eta_Q$ on the vector bundles $TQ$ and $N^\omega Q$ on $Q$; clearly, $\omega_Q=\iota^*(\omega_R)$.
Correspondingly, on $TQ$ and $N^\omega Q$
there are fiberwise symplectic volume forms
$\mathcal{S}_Q=:(1/\ell!)\,\omega_Q^{\wedge \ell}$ and $\mathcal{S}_{N^\omega Q}=:(1/c!)\,\eta_Q^{\wedge c}$, respectively.

Omitting symbols of pull-back for notational simplicity,
along $Q$ we have
$$
{\mu_R}=
\dfrac{1}{d!}\,{\omega_R}^d=\left(\dfrac{1}{\ell!}\,{\omega_Q}^{\wedge \ell}\right)\wedge 
\left(\dfrac{1}{c!}\,{\eta_Q}^{\wedge c}\right)=
{\mathcal{S}_Q}\wedge \mathcal{S}_{N^\omega Q}
$$
on $\bigwedge ^{2d}TR\cong \big[\bigwedge ^{2\ell}TQ\big]\wedge \big[\bigwedge ^{2c}N^\omega Q\big]$.

The oriented vector bundles $TQ$ and $N^\omega Q$ also have 
metric structures, induced by the Riemannian metric of $M$. Therefore,
they have Euclidean (Riemannian) volume forms $\mathcal{E}_Q$ and $\mathcal{E}_{N^\omega Q}$, respectively, defined
pointwise as above.

If $F$ is a complex submanifold, 
then $\mathcal{E}_Q= \mathcal{S}_Q$ and $\mathcal{E}_{N^\omega Q}= \mathcal{S}_{N^\omega Q}$.
For a general symplectic submanifold,
%$\mathcal{E}_F\neq \mathcal{S}_F$ and $\mathcal{E}_N\neq \mathcal{S}_N$.	
as a global version of (\ref{eqn:comparison_volume_form_vector_space}) we now have
\begin{equation}
 \label{eqn:comparison_volume_form_submanifold}
\mathcal{S}_Q=\zeta _Q\,\mathcal{E}_Q\,\,\,\,\,\,\,\mathrm{and}\,\,\,\,\,\,\,
\mathcal{S}_{N^\omega Q}=\zeta _{N^\omega Q}\cdot\mathcal{E}_{N^\omega Q},
\end{equation}
for unique $\mathcal{C}^\infty$ functions $\zeta_Q,\,\zeta_{N^\omega Q}:Q\rightarrow (0,+\infty)$.
Therefore along the symplectic submanifold $Q\subseteq R$
\begin{equation}
 \label{eqn:defn_zeta_F}
\mu_M=\mathcal{S}_Q\wedge\mathcal{S}_R=
\zeta_N\cdot\mathcal{S}_Q\wedge \mathcal{E}_{N^\omega Q}=\big(\zeta_Q\,\zeta_{N^\omega Q}\big)
\cdot\mathcal{E}_Q\wedge \mathcal{E}_{N^\omega Q}.
\end{equation}

Suppose $q\in Q$, and let $U\subseteq M$ be a sufficiently small open neighborhood
of $q$, so that we can find
local coordinate systems $\beta_m(\mathbf{u})=m+\mathbf{u}$, 
centered at
$m\in U$ and smoothly varying with $m$. 
Using Taylor expansion in $\mathbf{u}$, in additive notation we obtain for $q'\in Q\cap U$:
\begin{eqnarray}
 \label{eqn:defn_zeta_F_rescaled}
\mu_M\left(q'+\frac{\mathbf{u}}{\sqrt{k}}\right)&\sim&\zeta_{N^\omega Q}(q')\,
\left[1+\sum_{j\ge 1}k^{-j/2}\,A_j(q',\mathbf{u})\right]\cdot{S_Q}_{q'}\wedge E_{N_{q'}^\omega Q}\\
&=&\zeta_Q(q')\,\zeta_{N^\omega Q}(q')\cdot \left[1+\sum_{j\ge 1}k^{-j/2}\,B_j(q',\mathbf{u})\right]
\cdot{E_Q}_{q'}\wedge E_{N_{q'}^\omega Q}.
\nonumber
\end{eqnarray}
Here $A_j$ and $B_j$ are $\mathcal{C}^\infty$ in $q'\in U\cap Q$ and homogeneous polynomials 
of degree $j$ in $\mathbf{u}$.

\subsection{Trace asymptotic expansion at a fixed time}

The fixed-time asymptotic expansion for the trace that we shall
discuss presently is really a consequence of the local scaling asymptotics
 at fixed time in \cite{paoletti_intjm},
and in fact it was briefly touched upon in the introduction of that paper for the special
case of isolated fixed points; this case was also treated in
\cite{ch} by a different approach. We take up the issue here again 
in more detail and generality than in \cite{paoletti_intjm},
both to put things in 
perspective towards the trace scaling asymptotics to follow,
and for its obvious intrinsic interest. In the holomorphic case, the global benchmark is of
 course the classical Atyah-Singer Lefschetz fixed point formula; asymptotic
approaches based on 
local scaling expansions, encompassing compositions with Toeplitz operators 
and isotypic decompositions under Lie group actions,
were given in \cite{mz} and \cite{pao-trace}.

\begin{thm}
 \label{thm:trace_asymptotics_tau_0}
Let $\tau_0$ be a very clean period of $\phi^M$, and let $F_1, \dots,F_{\ell_{\tau_0}}$
be the connected components of $M_{\tau_0}\subseteq M$.
For $a=1,\ldots,\ell_{\tau_0}$,
let $2\,d_a$ be the real dimension of the symplectic submanifold $F_a$.
Then 
$$
\mathrm{trace}(U_{\tau_0,k})=\sum_{a=1}^{\ell_{\tau_0}}\mathcal{F}_a(\tau_0, k),
$$
where each summand $\mathcal{F}_a(\tau_0, k)$ is given by an asymptotic expansion
of the form
\begin{eqnarray}
 \label{eqn:trace_asymptotics_tau_0_ath_summand}
\mathcal{F}_a(\tau_0,k)&\sim&h_a^k\,\left(\frac k\pi\right)^{d_a}\,
\sum_{j\ge 0}k^{-j}\,F_{aj}(\tau_0),
\end{eqnarray}
with leading coefficient  
$$
F_{a0}(\tau_0)=2^d\,\int_{M_{\tau_0}}\dfrac{\varrho_{\tau_0}(p)}{\nu(\tau_0,p)}\,\zeta_{N_a}(p)\,
\det\left(\mathfrak{Q}_p^{\mathrm{nor}}\right)^{-1/2}\,\mathrm{dV}_F^s(p),
$$
where the latter factor in the integrand is as in Definition \ref{defn:normal_quadratic_form},
$$
\zeta_{N_a}=:\zeta_{N^\omega F_a}:F_a\longrightarrow (0,+\infty)
$$ 
is as in (\ref{eqn:defn_zeta_F}), and
we have denoted by 
$\mathrm{dV}_F^s=|\mathcal{S}_F|$ the density (or its measure)
associated to the symplectic volume form. 
\end{thm}

\begin{rem}
 In fact, 
$\mathcal{F}_a(\tau_0,k)$ is supported near $F_a$, 
in the sense that it vanishes (up to negligible contributions)
if $R_{\tau_0}$ is smoothing in a neighborhood $U$ of $F_a$ .
\end{rem}

\begin{proof}

The function $x\mapsto U_{\tau',k}(x,x)$ descends for any $\tau'\in \mathbb{R}$  
to a well-defined $\mathcal{C}^\infty$ function on $M$,
call it $\mathbf{U}_{\tau',k}$. Thus $\mathbf{U}_{\tau',k}(m)=U_{\tau',k}(x,x)$ if $m=\pi(x)$. Furthermore, 
\begin{equation}
\label{eqn:trace_evaluation_integral}
 \mathrm{trace}(U_{\tau',k})=\int _X\,U_{\tau',k}(x,x)\,\mathrm{dV}_X(x)
=\int _M\,\mathbf{U}_{\tau',k}(m)\,\mathrm{dV}_M(m).
\end{equation}

Now let us set $\tau'=\tau_0$ in (\ref{eqn:trace_evaluation_integral}).
By Theorem \ref{thm:local_scaling_asymptotics_time}, up to a negligible
contribution the integrand in (\ref{eqn:trace_evaluation_integral}) asymptotically
localizes in a shrinking neighborhood of $M_{\tau_0}$. Therefore, for $\tau'=\tau_0$
we can rewrite 
(\ref{eqn:trace_evaluation_integral}) asymptotically as
\begin{equation}
\label{eqn:trace_evaluation_integral_localized}
 \mathrm{trace}\big(U_{\tau_0,k}\big)\sim \sum_{a=1}^{\ell_{\tau_0}}\int_{M'_a}\beta_a(m)\,\mathbf{U}_{\tau_0,k}(m)\,\mathrm{d}V_M(m),
\end{equation}
where $M'_a\subseteq M$ is a tubular neighborhood of $F_a$, and $\beta_a\in \mathcal{C}^\infty_0(M_a')$
is identically $=1$ in a smaller tubular neighborhood $M_a''\Subset M_a'$ of $F_a$. 
We aim to estimate asymptotically the $a$-th summand in (\ref{eqn:trace_evaluation_integral_localized}).

For any $p_0\in F_a$, we can find an open neighborhood $U\subseteq F_a$ of $p_0$, and 
a smoothly varying family of Heisenberg local coordinates centered at points 
in $U'=\pi^{-1}(U)$,
which we denote
$\Gamma:U'\times (-\pi,\pi)\times B_{2d}(\mathbf{0},\delta)\longrightarrow X$. 
More precisely, we require that
for any $x\in U'$ the restriction 
$$
\gamma_x=:\Gamma(x,\cdot,\cdot):(-\pi,\pi)\times B_{2d}(\mathbf{0},\delta)\longrightarrow X$$ 
be a HLCS centered at $x$, that we denote additively by 
$\gamma_x(\theta,\mathbf{v})=x+(\theta,\mathbf{v})$. In particular, $r_\vartheta(x+\mathbf{v})=x+(\vartheta,\mathbf{v})$.
In addition, we shall assume without loss that
$$
x+(\theta+\vartheta,\mathbf{v})=\big(x+(\theta,\mathbf{0})\big)+(\vartheta,\mathbf{v}).
$$
This gives a meaning to the expression 
$p+\mathbf{v}$ when $p\in U$ and $\mathbf{v}\in \mathbb{R}^{2d}$
has suitably small norm (namely, $p+\mathbf{v}=:\pi\big(x+\mathbf{v}\big)$ if $p=\pi(x)$).

Let $N^sF_a$ be the symplectic normal bundle of $F_a$ (Definition \ref{defn:tangent_normal_spaces}), 
and let $\left.N^sF_a\right|_U$ be its
restriction to $U\subseteq F_a$. Let $\mathfrak{N}_U\subseteq \left.N^sF_a\right|_U$ be a suitably small open
neighborhood of the zero section. Then  
$\varsigma_U:(p,\mathbf{u})\in \mathfrak{N}_U\mapsto p+\mathbf{u}\in M$ is a diffeomorphism onto its image 
$\mathfrak{N}_U'=:\zeta_U(\mathfrak{N}_U)\subseteq M$, which is then a tubular neighborhood of $U$).

Set
$c_a=:d-d_a$, so that $2\,c_a$ is the (real) codimension of $F_a$ in $M$.
Perhaps after replacing $U$ by a smaller open neighborhood of $p_0$ in $F_a$, 
we may find an orthogonal local trivialization of $N^sF_a$ on $U$ (with respect to the Riemannian
metric $g$); thus
we may smoothly and unitarily identify $N^s_pF_a\cong \mathbb{R}^{2c_a}$ for $p\in U$. Accordingly, 
perhaps after restricting $\mathfrak{N}_U$, we shall identify
$\mathfrak{N}_U\cong U\times B_{2c_a}(\mathbf{0},\delta)$ for some $\delta>0$. Using this,
we may think of $\varsigma_U$ as as diffeomorphism
\begin{center}
 $\varsigma_U:U\times B_{2c_a}(\mathbf{0},\delta)\cong \mathfrak{N}_U'$ with 
$(p,\mathbf{u})\mapsto m=m(p,\mathbf{u})=:p+\mathbf{u}$.
\end{center}

When $U$ varies in a sufficiently fine finite open cover $(U_{aj})$ of $F_a$, 
the union of all the $\mathfrak{N}'_{aj}=:\mathfrak{N}'_{U_{aj}}$'s 
is an open neighborhood
of $F_a$, which we may assume without loss to be equal to $M'_a$ 
in (\ref{eqn:trace_evaluation_integral_localized}). 
Let $(\beta_{aj})$ be a partition of unity on $M'_a$ subordinate to the open cover$(\mathfrak{N}'_{aj})$. 
We can then rewrite the $a$-th summand in (\ref{eqn:trace_evaluation_integral_localized})  
as follows:
\begin{eqnarray}
 \label{eqn:a_th_summand_expanded}
\lefteqn{\int_{M'_a}\beta_a(m)\,\mathbf{U}_{\tau_0,k}(m)\,\mathrm{d}V_M(m)}\\
&=&\sum_j\int_{\mathfrak{N}'_{aj}}\beta_{aj}(m)\,\beta_a(m)\,\mathbf{U}_{\tau_0,k}(m)\,\mathrm{d}V_M(m)\nonumber\\
&=&\sum_j\int_{U_{aj}\times B_{2c_a}(\mathbf{0},\delta)}\beta_{aj}(p+\mathbf{u})\,\beta_a(p+\mathbf{u})\,\mathbf{U}_{\tau_0,k}(p+\mathbf{u})
\,\mathrm{dV}_M^{(j)}(p+\mathbf{u});\nonumber
\end{eqnarray}
here in the $j$-th summand 
$p+\mathbf{u}=\zeta_{U_j}(p,\mathbf{u})$, and
$$\mathrm{dV}_M^{(j)}(p+\mathbf{u})=:\zeta_{U_j}^*(\mathrm{dV}_M)(p,\mathbf{u}).$$

By Proposition \ref{prop:rapid_decrease}, only a negligible contribution
to the asymptotics is lost if the integrand in the $j$-th summand in (\ref{eqn:a_th_summand_expanded}) 
is multiplied by 
$\gamma_k(\mathbf{u})=:\gamma_1\left(k^{7/18}\,\|\mathbf{u}\|\right)$, $\gamma_1\in \mathcal{C}^\infty_0\left(\mathbb{R}\right)$
being a bump function
$\equiv 1$ in a neighborhood of the origin. The factor $\beta_a(p+\mathbf{u})$ may then be omitted.

If we now compose each $\zeta_{U_j}$ with rescaling in $\mathbf{u}$
by $k^{-1/2}$, (\ref{eqn:a_th_summand_expanded}) may be rewritten as follows: 
\begin{eqnarray}
\label{eqn:a_th_summand_integrated_recaled_u}
 \lefteqn{\int_{M'_a}\beta_a(m)\,\mathbf{U}_{\tau_0,k}(m)\,\mathrm{d}V_M(m)}\\
&\sim&k^{-c_a}\,\sum_j\int_{U_{aj}\times \mathbb{R}^{2c_a}}\gamma_1\left(k^{-1/9}\,\|\mathbf{u}\|\right)\,
\beta_{aj}\left(p+\frac{\mathbf{u}}{\sqrt{k}}\right)\nonumber\\
&&\,\,\,\,\,\,\,\,\,\,\,\,\,\,\,\,\,\,\,\,\,\,\,\,\,\,\,\,\,\,\,\,\,\,\,\,\,\,\,\,\,\,\,\cdot 
\mathbf{U}_{\tau_0,k}\left (p+\frac{\mathbf{u}}{\sqrt{k}}\right)\,\mathrm{dV}_M^{(j)}\left(p+\frac{\mathbf{u}}{\sqrt{k}}\right).
\nonumber
\end{eqnarray}

Given (\ref{eqn:defn_zeta_F_rescaled}), we have for each $j$
\begin{eqnarray}
\label{eqn:volume_form_rescaled_a}
%\lefteqn{
\mathrm{d}V_M^{(j)}\left(p+\frac{\mathbf{u}}{\sqrt{k}}\right)
 %=\left[\zeta_N(m)+\sum_{j\ge 1}k^{-j/2}\,\mathcal{A}_j(m,\mathbf{u})\right]\cdot
%\,\mathrm{d}\mathbf{u}\,\mathrm{dV}_F(m)
%}
&\sim&\zeta_{N^a}(p)\cdot \left[1+\sum_{h\ge 1}k^{-h/2}\,B_h^{(j)}(p,\mathbf{u})\right]\nonumber\\
&&\cdot (S_{Fa})_p\wedge \mathrm{d}\mathbf{u},
\end{eqnarray}
where we have written $N^a$ for $N^sF_a$ and $\mathrm{d}\mathbf{u}$ for
the standard volume density on
$\mathbb{R}^{2c_a}$, and $B_h^{(j)}:U_j\times \mathbb{R}^{2c_a}\rightarrow \mathbb{R}$ is homogeneous 
of degree $h$ in $\mathbf{u}$.
Similarly,
\begin{equation}
\label{eqn:beta_aj_rescaled}
 \beta_{aj}\left(p+\frac{\mathbf{u}}{\sqrt{k}}\right)\sim
\sum_{l\ge 0}k^{-l/2}\,\beta_{ajl}\left(p,\mathbf{u}\right),
\end{equation}
where $\beta_{ajl}\left(p,\mathbf{u}\right)$ is homogeneous in $\mathbf{u}$ of degree $l$,
and $\beta_{aj0}\left(p,\mathbf{u}\right)=\beta_{aj}(p)$.

Then, multiplying the asymptotic
expansions in (\ref{eqn:volume_form_rescaled_a}), (\ref{eqn:beta_aj_rescaled}) and 
Corollary \ref{cor:local_fixed_time_expansion},
we conclude that (\ref{eqn:a_th_summand_integrated_recaled_u}) may be rewritten
\begin{eqnarray}
\label{eqn:a_th_summand_integrated_iterated_integral_Lambda}
%\lefteqn{
\int_{M'_a}\beta_a(m)\,\mathbf{U}_{\tau_0,k}(m)\,\mathrm{d}V_M(m)%}\\
&\sim&
 h_a^k\,k^{-c_a}\,\left(\frac k\pi\right)^d\,\int _{F_a}
\,\Lambda_k(p)\,\mathcal{S}_{F_a}(p),
\end{eqnarray}
where $\Lambda_k$ has the form
\begin{eqnarray}
\label{eqn:defn_of_Lambda_k}
 \Lambda_k(p)=\int_{\mathbb{R}^{2c_a}}\gamma_1\left(k^{-1/9}\,\mathbf{u}\right)
e^{-\mathfrak{Q}_{\tau_0,p}(\mathbf{u})}\,
G_k(p,\mathbf{u})\,\mathrm{d}\mathbf{u};
\end{eqnarray}
here
\begin{eqnarray}
 \label{eqn:defn_G_k_p_u}
G_k(p,\mathbf{u})&\sim& \dfrac{2^d}{\nu(\tau_0,p)}\,\zeta_{N^a}(p)
\sum_{l\ge 0}k^{-l/2}\,c_{l}(p,\mathbf{u}),
\end{eqnarray}
and the $c_l$'s are polynomials in $\mathbf{u}$, of degree $\le 3\,l$ and parity $l$;
the leading term is 
\begin{eqnarray}
 \label{eqn:G_k_0}
 c_{0}(p,\mathbf{u})=\varrho_{\tau_0}(p).
\end{eqnarray}

We now aim to estimate $\Lambda_k$ asymptotically. Since in (\ref{eqn:defn_of_Lambda_k}) 
integration in $\mathrm{d}\mathbf{u}$ is over a ball centered at the origin and of radius $O\left(k^{1/9}\right)$,
the asymptotic expansion (\ref{eqn:defn_G_k_p_u}) may be integrated term by term. 
Given this, in view of
(\ref{eqn:normal_restriction_pos_def}) we only lose a rapidly decreasing contribution if in (\ref{eqn:defn_of_Lambda_k})
the cut-off is omitted and integration is taken over all of $\mathbb{R}^{2c_a}$.
Therefore, taking into account the previous assertion about the parity of the $c_l$'s,
\begin{eqnarray}
 \label{eqn:defn_Lambda_k_expanded}
 \Lambda_k(p)&\sim& \sum_{l\ge 0}k^{-l}\,\lambda_l(p),
\end{eqnarray}
where 
\begin{eqnarray}
\label{eqn:defn_lambda_l}
\lambda_l(p)&=:&\dfrac{2^d}{\nu(\tau_0,p)}\,\zeta_{N^a}(p)
\int_{\mathbb{R}^{2c_a}}\,c_{2l}(p,\mathbf{u})
\,e^{-\mathfrak{Q}_{\tau_0,p}(\mathbf{u})}\,\mathrm{d}\mathbf{u}.
\end{eqnarray}
In particular, the leading order term is
\begin{eqnarray}
\label{eqn:defn_lambda_0}
\lambda_0(p)&=:&\dfrac{2^d}{\nu(\tau_0,p)}\,\zeta_{N^a}(p)\,
\varrho_{\tau_0}(p)\cdot 
\int_{\mathbb{R}^{2c_a}}\,
\,e^{-\mathfrak{Q}_{\tau_0,p}(\mathbf{u})}\,\mathrm{d}\mathbf{u}\nonumber\\
&=&\dfrac{2^d}{\nu(\tau_0,p)}\,\zeta_{N^a}(p)\,
\varrho_{\tau_0}(p)\,\pi^{c_a}\,\det\left(\mathfrak{Q}_p^{\mathrm{nor}}\right)^{-1/2}.
\end{eqnarray}

Inserting this in (\ref{eqn:a_th_summand_integrated_iterated_integral_Lambda}) completes the proof
of Theorem \ref{thm:trace_asymptotics_tau_0}.

\end{proof}

\section{Trace asymptotics at rescaled times}

\subsection{Morse-Bott periods}

We next aim to discuss global scaling asymptotics for traces, the scaling being now with respect
to the time variable. The result that we shall discuss presently applies to a more restrictive class
of periods than very clean ones, which we now define.

\begin{defn}
\label{defn:morse_bott_periods}
We shall say that $\tau_0$ is a \textit{Morse-Bott (very) clean period}
  of $\phi^M$ if
  \begin{enumerate}
   \item $\tau_0$ is a (very) clean period of $\phi^M$
(Definition \ref{defn:very_clean_periods});
   \item the restriction of $f$ to $M_{\tau_0}$, $f_{\tau_0}=:\left.f\right|_{M_{\tau_0}}:M_{\tau_0}\rightarrow M_{\tau_0}$, is a Morse-Bott
   function.
  \end{enumerate}
\end{defn}

Albeit quite restrictive, these conditions are satisfied in many natural situations.

\begin{exmp}
 If $f:M\rightarrow M$ is a Morse-Bott function, then $0\in \mathbb{R}$ is a Morse-Bott
 very clean period of $f$.
\end{exmp}

\begin{exmp}
Let $(N,\eta)$ be any compact symplectic manifold.
If $\delta\ge 1$ is an integer, suppose that $\mu:\mathbf{T}^\delta\times M\rightarrow M$ is an
Hamiltonian action of the compact $\delta$-dimensional torus $\mathbf{T}^\delta$ on $(N,\eta)$, with moment map
$\Phi:N\rightarrow \mathfrak{t}^*\cong \mathbb{R}^\delta$. If $\xi\in \mathfrak{t}$ (the Lie algebra of $\mathbf{T}$),
set $f=:\Phi_\xi=\langle \Phi,\xi\rangle :N\rightarrow \mathbb{R}$; then the Hamiltonian flow of
$f$ is simply the restriction of $\mu$ to the 1-parameter subgroup generated by $\xi$. 
Then $f$ is, for any choice of $\xi$,
a Morse-Bott function and its critical submanifolds are all symplectic, 
with even Morse indexes and coindexes.
Furthermore, for any $g\in \mathbf{T}^\delta$ the fixed locus $N_g\subseteq N$ of $\mu_g$ is a 
$\mathbf{T}$-invariant compact symplectic submanifold, and so the restriction of $f$ to it is a Morse-Bott function
(\cite{atiyah}, \cite{mcduff-salamon},
\cite{nicolaescu}, \S 3.5).
\end{exmp}

\begin{defn}
 \label{defn:technical_notation}
If $\tau_0\in \mathbb{R}$ is a Morse-Bott very clean period of $\phi^M$,
we shall adopt the following notation.
As before,
$(F_a)$ will be the connected components of the fixed locus
$M_{\tau_0}$, $a=1,\ldots,\ell_{\tau_0}$. For each $a$, 
\begin{itemize}
 \item $d_a=:\frac 12\,\dim (F_a)\in \mathbb{Z}$;
\item $f_a:F_a\rightarrow \mathbb{R}$ denotes the restriction
of $f$ to $F_a$.
\item $C_{ab}\subseteq F_a$, ($1\le b\le s_a$), is the 
connected critical submanifolds of
$f_a$ in $F_a$.
\end{itemize}
Furthermore,
for each pair $(a,b)$ with $a=1,\ldots,\ell_{\tau_0}$ and $b=1,\ldots,s_a$:
\begin{itemize}
 \item $d_{ab}=:\frac 12\,\dim(C_{ab})\in \frac 12\,\mathbb{Z}$.
\item $f_{ab}\in \mathbb{R}$ is the constant value of $f$ on $C_{ab}$.
\item $H_q(f)$ is the (non-degenerate) transverse Hessian of $f_a$ at $q\in C_{ab}$.
\item $q\in C_{ab}\mapsto \det \big(H_q(f)\big)$ denotes the determinant of
$H_q(f)$ with respect to any orthonormal basis of $N^g_q(C_{ab}/F_a)$.
\item $\sigma_{ab}\in \mathbb{Z}$ is the constant signature of $H_q(f)$ along $C_{ab}$. 
\end{itemize}

\end{defn}

\begin{thm}
 \label{thm:global_trace_asymptotics_tau_0_morse_bott}
Suppose that $\tau_0\in \mathbb{R}$ is a Morse-Bott very clean period of $\phi^M$,
and adopt the notation in Definition \ref{defn:technical_notation}. Fix $C>0$.
Then,
uniformly for
\begin{equation}
 \label{eqn:bound_on_tau}
C\,k^{-\frac{1}{9}}<|\tau|<C\,k^{\frac{1}{9}},
\end{equation}
we have
$$
\mathrm{trace}\big(U_{\tau_0+\frac{\tau}{\sqrt{k}}}\big)=\sum_{a=1}^{\ell_{\tau_0}}\sum_{b=1}^{s_a}\mathcal{L}_{ab}(k,\tau),
$$
where for $k\rightarrow +\infty$ each summand $\mathcal{L}_{ab}(k,\tau)$
is given by an asymptotic expansion of the form
\begin{eqnarray}
 \label{eqn:global_trace_asymptotics_tau_0_morse_bott_ath}
\mathcal{L}_{ab}(k,\tau)&\sim&h_a^k\cdot \dfrac{k^{\frac 12\,(d_a+d_{ab})}}{\pi^{d_a}}\cdot
\left(\dfrac{2\pi}{\tau}\right)^{d_a-d_{ab}}
\,e^{i\sqrt{k}\,\tau f_{ab}+\frac i4\pi\,\sigma_{ab}}\nonumber\\%\cdot h^k
&&\cdot \sum_{l,r\ge 0 }k^{-l/2}\left(\tau\sqrt{k}\right)^{-r}\,\mathcal{A}_{ablr}(\tau),
\end{eqnarray}
with $\mathcal{A}_{ablr}(\tau)$ a polynomial in $\tau$ of degree $\le 3(l+r)$.
The leading order coefficient is given by
\begin{eqnarray}
\mathcal{A}_{ab00}&=&\int_{C_{ab}}\,\varrho_{\tau_0}(q)\,
\dfrac{2^d}{\nu(\tau_0,q)}\,\zeta_{F_a}(q)\,\zeta_{N_a}(q)\nonumber\\
&&\cdot\det\left(\mathfrak{Q}_q^{\mathrm{nor}}\right)^{-1/2}\,|\det H_q(f)|^{-1/2}\,
\mathcal{E}_{C_{ab}}(q).
\end{eqnarray}
Furthermore, $\mathcal{L}_{ab}$ is localized in the neighborhood of $C_{ab}$, in the sense that it is negligible
as soon as $R_{\tau'}$ is smoothing near $C_{ab}$, for $\tau'$ near $\tau_0$.
\end{thm}

\begin{rem}
 \label{rem:why_expansion}
To see why (\ref{eqn:global_trace_asymptotics_tau_0_morse_bott_ath}) is indeed an asymptotic expansion
for $k\rightarrow +\infty$, let us
write
 $$
\sum_{l,r\ge 0 }k^{-l/2}\left(\tau\sqrt{k}\right)^{-r}\,\mathcal{A}_{ablr}=
\sum_{s\ge 0}\left[\sum_{l+r=s }k^{-l/2}\left(\tau\sqrt{k}\right)^{-r}\,\mathcal{A}_{ablr}(\tau)\right],
$$
and remark that given that $C\,k^{-\frac{1}{9}}<|\tau|$ by
(\ref{eqn:bound_on_tau}) if $l+r=s$ then
$$
k^{l/2}\left(|\tau|\sqrt{k}\right)^{r}\ge C^r\,k^{\frac 12\,(l+r)-\frac{1}{9}\,r}\ge C^r\,k^{\frac{7}{18}\,s}.
$$
Therefore, since each $\mathcal{A}_{ablr}(\tau)$ has degree $\le 3s$ in $\tau$, and we are also assuming
$|\tau|<C\,k^{\frac{1}{9}}$ by
(\ref{eqn:bound_on_tau}), we have
$$
\left|\sum_{l+r=s }k^{-l/2}\left(\tau\sqrt{k}\right)^{-r}\,\mathcal{A}_{ablr}(\tau)\right|\le
C_s\,k^{-\frac{7}{18}\,s}\,k^{\frac{1}{9}\,3s}=C_s\,k^{-\frac{1}{18}\,s}.
$$
\end{rem}

\begin{proof}
Let us backtrack to the proof of Theorem \ref{thm:trace_asymptotics_tau_0},
and set $\tau'=\tau_k=:\tau_0+\tau/\sqrt{k}$ in (\ref{eqn:trace_evaluation_integral}). Then (\ref{eqn:trace_evaluation_integral_localized}),
(\ref{eqn:a_th_summand_expanded}), and (\ref{eqn:a_th_summand_integrated_recaled_u}) continue to hold, with $\tau_0$
replaced by $\tau_k$.

However, instead of Corollary \ref{cor:local_fixed_time_expansion},
we now need to use Corollary \ref{cor:local_scaling_asymptotics_time_fixed_point_case}.
Since in our construction
$\mathbf{u}\in N_p^\omega F_a$ while $\upsilon_f(p)\in T_pF_a$ because $F_a$
is $\phi^M$-invariant,
we have $\omega_p\big(\upsilon_f(p),\mathbf{u}\big)=0$.
Therefore, in place of (\ref{eqn:a_th_summand_integrated_iterated_integral_Lambda}) we have
\begin{eqnarray}
\label{eqn:a_th_summand_integrated_iterated_integral_Lambda_tau}
\lefteqn{
\int_{M'_a}\beta_a(m)\,\mathbf{U}_{\tau_k,k}(m)\,\mathrm{d}V_M(m)}\nonumber\\
&\sim&
 h_a^k\,k^{-c_a}\,\left(\frac k\pi\right)^d\,\int _{F_a}
\,\Lambda_k(p,\tau)\,\mathcal{S}_{F_a}(p),
\end{eqnarray}
where now
\begin{eqnarray}
\label{eqn:defn_of_Lambda_k_tau}
 \Lambda_k(p,\tau)=:e^{i\tau\,\sqrt{k}\,f(p)}\,\int_{\mathbb{R}^{2c_a}}\gamma_1\left(k^{-1/9}\,\mathbf{u}\right)
e^{
\mathcal{S}_{\tau_0,p}\big(\mathbf{u},\tau\,\upsilon_f(p)+\mathbf{u}\big)}\,
G_k(p,\mathbf{u},\tau)\,\mathrm{d}\mathbf{u};
\end{eqnarray}
here
\begin{eqnarray}
 \label{eqn:defn_G_k_p_u}
G_k(p,\mathbf{u},\tau)&\sim& \dfrac{2^d}{\nu(\tau_0,p)}\,\zeta_{N^a}(p)
\sum_{l\ge 0}k^{-l/2}\,c_{l}(p,\mathbf{u},\tau),
\end{eqnarray}
and the $c_l$'s are polynomials in $(\mathbf{u},\tau)$, of degree $\le 3\,l$;
the leading term is 
\begin{eqnarray}
 \label{eqn:G_k_0}
 c_{0}(p,\mathbf{u})=\varrho_{\tau_0}(p).
\end{eqnarray}
We have
on the exponent
the estimate (\ref{eqn:normal_restriction_pos_def-hamiltonian}) in Remark \ref{rem:positive_definite_hamiltonian}.

Therefore, in place of (\ref{eqn:defn_Lambda_k_expanded}) we now have 
\begin{eqnarray}
 \label{eqn:defn_Lambda_k_expanded_tau}
 \Lambda_k(p,\tau)&\sim& \sum_{l\ge 0}k^{-l/2}\,\lambda_l(p,\tau),
\end{eqnarray}
where 
\begin{eqnarray}
\label{eqn:defn_lambda_l_tau}
\lambda_l(p,\tau)&=:&\dfrac{2^d}{\nu(\tau_0,p)}\,\zeta_{N^a}(p)\,e^{i\tau\,\sqrt{k}\,f(p)}\nonumber\\
&&\cdot \int_{\mathbb{R}^{2c_a}}\,e^{\mathcal{S}_{\tau_0,p}\big(\mathbf{u},\tau\,\upsilon_f(p)+\mathbf{u}\big)}\,
c_{l}(p,\mathbf{u},\tau)
\,\mathrm{d}\mathbf{u}\nonumber\\
&=&e^{i\tau\,\sqrt{k}f(p)}\,\dfrac{2^d}{\nu(\tau_0,p)}\,\zeta_{N^a}(p)\,\,K_l(p,\tau),
\end{eqnarray}
with 
\begin{equation}
\label{eqn:defn_K_l}
 K_l(p,\tau)=:\int_{\mathbb{R}^{2c_a}}\,e^{\mathcal{S}_{\tau_0,p}\big(\mathbf{u},\tau\,\upsilon_f(p)+\mathbf{u}\big)}\,
c_{l}(p,\mathbf{u},\tau)
\,\mathrm{d}\mathbf{u}.
\end{equation}

Using this in (\ref{eqn:a_th_summand_integrated_iterated_integral_Lambda_tau}), we obtain
\begin{eqnarray}
\label{eqn:a_th_summand_integrated_iterated_integral_Lambda_tau_expanded}
\int_{M'_a}\beta_a(m)\,\mathbf{U}_{\tau_k,k}(m)\,\mathrm{d}V_M(m)\sim
h_a^k\,k^{-c_a}\,\left(\frac k\pi\right)^d\,
 \sum_{l\ge 0}k^{-l/2}\,H_{al}(k),
\end{eqnarray}
with
\begin{equation}
 \label{eqn:defn_H_al}
H_{al}(k,\tau)=:
\int _{F_a}
\,e^{i\tau\,\sqrt{k}f(p)}\,\dfrac{2^d}{\nu(\tau_0,p)}\,\zeta_{N^a}(p)\,\,K_l(p,\tau)\,\mathcal{S}_{F_a}(p).
\end{equation}

In the range (\ref{eqn:bound_on_tau}), $\tau\,\sqrt{k}\rightarrow +\infty$ for $k\rightarrow +\infty$, 
and therefore
$H_{al}(k)$ may be interpreted as an oscillatory integral in $\tau\,\sqrt{k}$ with real phase $f$.
Since we are assuming $|\tau|<C\,k^{1/9}$ we also have
$|\tau|^2<C\,\big(\tau\,\sqrt k\big)^\delta$, with say $\delta=2/5<1/2$. 
Call $G_l$ the amplitude in (\ref{eqn:defn_H_al}); then we see from (\ref{eqn:defn_lambda_l_tau})
that in any local coordinate system on $F_a$
\begin{equation}
\label{eqn:stationary_phase_condition_amplitude}
 \left|\partial^\alpha G_l\right|=O\left(|\tau|^{3l+2\,|\alpha|}\right)=O\left(|\tau|^{3l+\frac 25\,|\alpha|}\right).
\end{equation}

Therefore, we may apply the first part
the stationary phase Lemma \cite{duist}, and conclude that the integral (\ref{eqn:defn_H_al})
asymptotically localizes in the neighborhood of the critical manifolds $C_{ab}$. More precisely, for 
every $a=1,\ldots,\ell_{\tau_0}$ and $b=1,\ldots,s_a$ let $F_{ab}'\subseteq F_a$ be an open neighborhood of
$C_{ab}$ in $F_a$, and let $\gamma_{ab}:F_a\rightarrow [0,+\infty)$ be a bump function, compactly
supported in $F'_{ab}$ and $\equiv 1$ on some smaller neighborhood $F''_{ab}\Subset F'_{ab}$.
Then as $k\rightarrow +\infty$ we have
\begin{equation}
\label{eqn:defn_H_al_asymp}
H_{al}(k,\tau)=\sum_{b=1}^{s_a}H_{abl}(k,\tau),
\end{equation}
where
\begin{equation}\label{eqn:defn_H_abl_asymp}
 H_{abl}(k,\tau)=:\int _{F'_{ab}}\gamma_{ab}(p)\,
\,e^{i\tau\,\sqrt{k}f(p)}\,\dfrac{2^d}{\nu(\tau_0,p)}\,\zeta_{N^a}(p)\,\,K_l(p,\tau)\,\mathcal{S}_{F_a}(p).
\end{equation}

Arguing as in the proof of Theorem \ref{thm:trace_asymptotics_tau_0}, we can furthermore find a 
$\mathcal{C}^\infty$ finite
partition of unity 
$\delta_{abj}:F_{ab}'\rightarrow [0,+\infty)$ subordinate to an open cover $F_{abj}'$, such that
the following holds. For each $j$, let $C_{abj}=C_{ab}\cap F_{abj}$; thus $(C_{abj})$ is an open cover of
$C_{ab}$. Let $B_{2(d_a-d_{ab})}(\mathbf{0},\delta)\subseteq \mathbb{R}^{2(d_a-d_{ab})}$
be the open ball centered at the origin and of radius $\delta$. Then 
for each $j$ and some $\delta>0$ there is a diffeomorphism 
$\Phi_{abj}:C_{abj}\times B_{2(d_a-d_{ab})}(\mathbf{0},\delta)\rightarrow F_{abj}'$ of the form
$(q,\mathbf{n})\mapsto p(q,\mathbf{n})= q+\mathbf{n}=:\varphi_q^{(j)}(\mathbf{n})$, where $\varphi_q^{(j)}$ is a local
coordinate system on $F_a$ centered at $q$; furthermore, we shall require that for every $q\in C_{abj}$ 
the differential of
$d_\mathbf{0}\zeta_q$ induces a unitary isomorphism $\mathbb{R}^{2d_a}\cong T_qF_a$,
under which $\{\mathbf{0}\}\times \mathbb{R}^{2(d_a-d_{ab})}
\cong N^g _q(C_{ab}/F_a)$; here by $N^g _q(C_{ab}/F_a)$ we denote the Riemannian normal space to
$C_{ab}$ at $q$, as a submanifold of $F_a$.

Let us set $B=:B_{2(d_a-d_{ab})}(\mathbf{0},\delta)$.
Recalling that $\mathcal{S}_{F_a}=\zeta_{F_a}\cdot \mathcal{E}_{F_a}$,  
we can then rewrite $H_{abl}(k,\tau)$ in (\ref{eqn:defn_H_abl_asymp}) as follows:
\begin{eqnarray}
\label{eqn:defn_H_al_asymp_1}
\lefteqn{
H_{abl}(k,\tau)}
\\
&=&\sum_j\int _{C_{abj}}\int_{B}\,e^{i\tau\,\sqrt{k}f(q+\mathbf{n})}\,F_{abj}(q+\mathbf{n})\,\,Z_a(q+\mathbf{n},\tau)\,
\,\mathcal{E}_{F_a}(q+\mathbf{n}),\nonumber
\end{eqnarray}
where
$$
F_{abj}=:\gamma_{ab}\cdot \delta_{abj},\,\,\,\,\,\,\,\,
Z_{al}=:\zeta_{N^a}\cdot \zeta_{F_a}\,\,\dfrac{2^d}{\nu (\cdot,\tau_0)}\,\cdot K_l,
$$
and in the $j$-th summand we use $\Phi_{abj}$ to give a meaning to $q+\mathbf{n}$. Also,
\begin{equation}
 \label{eqn:comparison_volume_form_submanifold_1}
\mathcal{E}_{F_a}(q+\mathbf{n})=:\Phi_{abj}^*(\mathcal{E}_{F_a})(q+\mathbf{n})=\mathcal{V}_j(q,\mathbf{n})
\,| \mathrm{d}\mathbf{n}|\,
\mathcal{E}_{C_{ab}}(q),
\end{equation}
and by construction
$\mathcal{V}_j(q,\mathbf{n})=1+O(\mathbf{n})$.
Therefore, we can rewrite (\ref{eqn:defn_H_al_asymp_1}) as
\begin{eqnarray}
\label{eqn:defn_H_al_asymp_2}
\lefteqn{
H_{abl}(k,\tau)}
\\
&=&\sum_j\int _{C_{abj}}\left[\int_{B}\,e^{i\tau\,\sqrt{k}f(q+\mathbf{n})}\,\Psi_{abj}(q+\mathbf{n},\tau)
\,| \mathrm{d}\mathbf{n}|\right]\,\mathcal{E}_{C_{ab}}(q),\nonumber
\end{eqnarray}
where
\begin{equation}
 \label{eqb:defn_Psi_ablj}
\Psi_{ablj}=:F_{abj}\,Z_{al}\,\mathcal{V}_j
\end{equation}

By the Morse-Bott assumption, if $q\in C_{abj}$ the function $\mathbf{n}\in B\mapsto f(q+\mathbf{n})$ has
a non-degenerate critical point at the origin. Given this and (\ref{eqn:stationary_phase_condition_amplitude}), 
we may apply the second part of the stationary phase Lemma and obtain, with the
notation in Definition \ref{defn:technical_notation}, an asymptotic expansion of the form 
\begin{eqnarray}
\label{eqn:defn_H_al_asymp_3}
\lefteqn{
\int_{B}\,e^{i\tau\,\sqrt{k}f(q+\mathbf{n})}\,\Psi_{ablj}(q+\mathbf{n})
\,| \mathrm{d}\mathbf{n}|}\\
&\sim&e^{i\tau\,\sqrt{k}f_{ab}}\,\left(\frac{2\pi}{\tau\,\sqrt{k}}\right)^{d_a-d_{ab}}\,
\big|\det \big(H_q(f)\big)\big|^{-1/2}\,e^{i\,\frac{\pi}{4}\,\sigma_{ab}}\nonumber\\
&&\cdot \sum_{r=0}^{+\infty}\left(\tau\,\sqrt{k}\right)^{-r}\,\Psi_{abljr}(q,\tau),\nonumber
\end{eqnarray}
where $\Psi_{abljr}=:R_r(\Psi_{ablj})/r!$ and $R_r$ is a differential
operator of degree $2r$ in the $\mathbf{n}$-variables,
defined in terms of the transverse Hessian and the third order remainder
of $f_a$ at $q$. 
In particular, $R_0$ is the identity. 

It is then clear from this and (\ref{eqn:defn_K_l}) that $\Psi_{abljr}$ has the form
\begin{equation}
\label{eqn:psi_abljr}
 \Psi_{abljr}(p,\tau)=\int_{\mathbb{R}^{2c_a}}\,e^{\mathcal{S}_{\tau_0,p}\big(\mathbf{u},\tau\,\upsilon_f(p)+\mathbf{u}\big)}\,
c_{abljr}(p,\mathbf{u},\tau)
\,\mathrm{d}\mathbf{u}\,\,\,\,\,\,\,\,\,(p\in F'_{abj}),
\end{equation}
with $c_{abljr}(p,\mathbf{u},\tau)$ a polynomial in $\tau$ of degree $\le 3l+2r\le 3(l+r)$.

\begin{lem}
 If $q\in C_{ab}$, then $d_qf=0$; equivalently, $\upsilon_f(q)=0$. 
\label{lem:critical_locus_ab}
\end{lem}

\begin{proof}
By construction, $d_qf_{a}=0$. By assumption, $F_a$ is a connected symplectic submanifold
of $M$; let $\upsilon^{(a)}_f\in \mathfrak{X}(F_a)$ be the Hamiltonian vector field of $f_{a}$.
Thus $\upsilon^{(a)}_f(q)=0$.

On the other hand, $M_{\tau_0}$ is invariant under the flow $\phi^M$ generated by
$\upsilon_f$, and therefore
so is every connected component $F_a$ of $M_{\tau_0}$. 
Hence $\upsilon_f$ is tangent to $F_a$, so $\upsilon_f=\upsilon^{(a)}_f$ along
$F_a$.
In particular, $\upsilon_f=0$ on $C_{ab}$.
\end{proof}

By (\ref{eqn:psi_abljr}) and Lemma \ref{lem:critical_locus_ab},
for $q\in C_{ab}$ we have
\begin{equation}
 \label{eqn:defn_K_l_ab}
 \Psi_{abljr}(q,\tau)=\int_{\mathbb{R}^{2c_a}}\,e^{-\mathfrak{Q}_{\tau_0,p}(\mathbf{u})}\,
c_{abljr}(q,\mathbf{u},\tau)
\,\mathrm{d}\mathbf{u},
\end{equation}
which is a polynomial in $\tau$ of degree $\le 3(l+r)$.

Going back to (\ref{eqn:defn_H_al_asymp_2}), we obtain
\begin{eqnarray}
 \label{eqn:defn_H_al_asymp_2}
\lefteqn{
H_{abl}(k,\tau)}
\\
&\sim&e^{i\tau\,\sqrt{k}f_{ab}}\,\left(\frac{2\pi}{\tau\,\sqrt{k}}\right)^{d_a-d_{ab}}\,
e^{i\,\frac{\pi}{4}\,\sigma_{ab}}\nonumber\\
&&\cdot \sum_{r=0}^{+\infty}\left(\tau\,\sqrt{k}\right)^{-r}\,
\int_{C_{ab}}\,\big|\det \big(H_q(f)\big)\big|^{-1/2}\,\Psi_{ablr}(q,\tau)\,\mathcal{E}_{C_{ab}}(q),\nonumber
\end{eqnarray}
where $\Psi_{ablr}=:\sum_j\Psi_{abljr}$. Using this in (\ref{eqn:a_th_summand_integrated_iterated_integral_Lambda_tau_expanded})
and (\ref{eqn:defn_H_al_asymp}), we obtain
\begin{eqnarray}
 \label{eqn:a_th_summand_final_expansion}
\lefteqn{\int_{M'_a}\beta_a(m)\,\mathbf{U}_{\tau_k,k}(m)\,\mathrm{d}V_M(m)}\\
&\sim&h_a^k\,\sum_{b=1}^{s_a}e^{i\tau\,\sqrt{k}f_{ab}+i\,\frac{\pi}{4}\,\sigma_{ab}}\cdot
\,\dfrac{k^{\frac 12(d_a+d_{ab})}}{\pi^{d_a}}\,\left(\frac{2\pi}{\tau}\right)^{d_a-d_{ab}}\nonumber\\
&&\cdot \sum_{l,r\ge 0}k^{-l/2}\,\left(\tau\,\sqrt{k}\right)^{-r}\,
\left[\frac{1}{\pi^{c_a}}\,\int_{C_{ab}}\,\big|\det \big(H_q(f)\big)\big|^{-1/2}\,\Psi_{ablr}(q,\tau)\,\mathcal{E}_{C_{ab}}(q)\right].
\nonumber
\end{eqnarray}

We see from (\ref{eqn:a_th_summand_final_expansion}) that (\ref{eqn:global_trace_asymptotics_tau_0_morse_bott_ath}) 
holds with 
\begin{equation}
\label{eqn:computation_A_ablr}
 \mathcal{A}_{ablr}(\tau)=:\pi^{-c_a}\,\int _{C_{ab}}\big|\det \big(H_q(f)\big)\big|^{-1/2}\,\Psi_{ablr}(q,\tau)\,\mathcal{E}_{C_{ab}}(q).
\end{equation}

Regarding the leading order term, since $c_{0}(p,\mathbf{u},\tau)=\varrho_{\tau_0}(p)$, from (\ref{eqn:defn_K_l}) we get
\begin{equation}
\label{eqn:defn_K_0}
 K_0(q,\tau)=:\varrho_{\tau_0}(q)\,\int_{\mathbb{R}^{2c_a}}\,e^{\mathcal{S}_{\tau_0,q}(\mathbf{u},\mathbf{u})}\,
\,\mathrm{d}\mathbf{u}=\varrho_{\tau_0}(q)\,\pi^{c_a}\,\det\left(\mathfrak{Q}_q^{\mathrm{nor}}\right)^{-1/2}.
\end{equation}
Pairing this with (\ref{eqb:defn_Psi_ablj}) and (\ref{eqn:computation_A_ablr}), we obtain 
\begin{eqnarray}
\label{eqn:computation_A_ab00}
\mathcal{A}_{ab00}(\tau)
&=&\int _{C_{ab}}\zeta_{N^a}(q)\, \zeta_{F_a}(q)\,\dfrac{2^d}{\nu (q,\tau_0)}\,\varrho_{\tau_0}(q)\\
&&\,\,\,\,\,\,\,\,\,\,\,\,\,\,
\cdot \big|\det \big(H_q(f)\big)\big|^{-1/2}\,\det\left(\mathfrak{Q}_q^{\mathrm{nor}}\right)^{-1/2}\,\mathcal{E}_{C_{ab}}(q).
\nonumber\end{eqnarray}
\end{proof}

\end{document}